\documentclass[12pt,a4paper]{amsart}

\usepackage[abbrev]{amsrefs}
\usepackage{amssymb}

\theoremstyle{plain}
  \newtheorem{thm}{Theorem}[section]
  \newtheorem{lem}[thm]{Lemma}
  
  \newtheorem{cor}[thm]{Corollary}
  \newtheorem{prop}[thm]{Proposition}
  \newtheorem{clm}[thm]{Claim}

\theoremstyle{definition}
  \newtheorem{defn}[thm]{Definition}
  
  \newtheorem{ex}[thm]{Example}

\theoremstyle{remark}
  \newtheorem{rem}[thm]{Remark}

  \newtheorem*{ack}{Acknowledgment}

\numberwithin{equation}{section}

\newcommand{\cA}{\mathcal{A}}

\newcommand{\cC}{\mathcal{C}}

\newcommand{\cF}{\mathcal{F}}
\newcommand{\cM}{\mathcal{M}}

\DeclareMathOperator{\supp}{supp}

\DeclareMathOperator{\Ker}{Ker}
\DeclareMathOperator{\Image}{Im}

\DeclareMathOperator{\vol}{vol}

\DeclareMathOperator{\graph}{graph}
\newcommand{\ce}{\text{\rm ce}}
\newcommand{\mce}{\text{\rm mce}}
\newcommand{\wmce}{\text{\rm wmce}}
\newcommand{\R}{\mathbb{R}}
\newcommand{\Z}{\mathbb{Z}}
\newcommand{\id}{\textrm{id}}

\makeatletter
\@namedef{subjclassname@2020}{
\textup{2020} Mathematics Subject Classification}
\makeatother

\begin{document}

\title[Coarse geometry of mm-spaces]
{Coarse geometry of metric measure spaces}
\thanks{This work was supported by JSPS KAKENHI Grant Numbers 24K06729 and 24K06714.}
\author{Takayuki Okuda}
\address{(T. Okuda) Graduate School of Advanced Science and Engineering, Hiroshima University, 1-3-1 Kagamiyama, Higashi-Hiroshima City, Hiroshima, 739-8526, Japan}
\email{okudatak@hiroshima-u.ac.jp}
\author{Takashi Shioya}
\address{(T. Shioya) Mathematical Institute, Tohoku University, Sendai 980-8578, Japan}
\email{shioya@math.tohoku.ac.jp}
\begin{abstract}
Using ideas from optimal transport theory, we introduce a notion of measured coarse equivalence for metric measure spaces and define a corresponding variant of the uniformly finite homology of Block and Weinberger, called weighted $\ell^\infty$ homology, for large-scale doubling metric measure spaces.   We prove that this homology is invariant under measured coarse equivalence and that the vanishing of its zeroth homology is equivalent to weighted non-amenability.  The proofs combine techniques from optimal transport theory and the disintegration of measures. 
\end{abstract}
\maketitle

\tableofcontents

\section{Introduction}

Coarse geometry is the study of the large-scale structure of metric spaces (or, more generally, coarse spaces). It is an active area of research with close connections to geometric group theory and the topology of manifolds.
We refer to the books \cites{NY:LSgeom, R:lect} by Nowak--Yu and Roe
for coarse geometry.

In this paper, we develop a theory of coarse geometry of metric measure spaces, where an mm-space (or \emph{mm-space} for short) means a complete separable metric space equipped with a fully supported boundedly finite Borel measure.
Our motivation is as follows.
In coarse geometry, volume growth is one of the fundamental notions.
For an mm-space, it is natural to consider "measure growth", namely, the growth of the measure of metric balls.
It is also natural to introduce a notion of weighted amenability of an mm-space, since the amenability of a metric space is defined by the failure of an isoperimetric inequality.  In fact, such a definition has already appeared in \cite{GT:meas-scaling}.
Moreover, amenability as a metric space is equivalent to the non-existence of a Ponzi scheme.  It seems natural to formulate a Ponzi scheme on an mm-space in the language of optimal transport theory.
A hint for such a formulation can be found in \cite{My}.

To develop a theory of coarse geometry for mm-spaces, one of the first major problems is to formulate an appropriate notion of coarse equivalence between mm-spaces.
A naive approach is to require the existence of a measure-preserving map.
However, such an approach does not work well for establishing coarse equivalences between discrete and continuous spaces.
This is analogous to the well-known limitation of the Monge problem in optimal transport, which is formulated in terms of transport maps and therefore cannot directly relate discrete and continuous measures.
Kantorovich introduced the notion of a coupling (or transport plan) between measures to overcome the limitation and formulated a relaxed version of the transport problem, now known as the Kantorovich problem.
Inspired by the Kantorovich relaxation in optimal transport theory, we formulate a notion of coarse equivalence between mm-spaces in terms of couplings of measures.

We say that two mm-spaces $X$ and $Y$ are \emph{weakly measured coarsely equivalent} and write $X \simeq_\wmce Y$ if there exist Borel measures $\nu_X$ and $\nu_Y$ that are uniformly comparable to $\mu_X$ and $\mu_Y$, respectively, together with a coupling $\sigma$ between $\nu_X$ and $\nu_Y$, such that the support of $\sigma$ is a coarse correspondence (equivalently, bi-bornologous) (see Definition \ref{defn:w-mce} for precise definition).
We call such a coupling a \emph{measured coarse correspondence}.

Based on the idea of Coulhon and Saloff-Coste \cite{CS:var-inf}, Tessera \cite{Ts:large-sob} introduced a notion of \emph{large-scale equivalence} between large-scale doubling mm-spaces in terms of the measures of metric balls.
Here, ``large-scale doubling'' means that the doubling condition holds only for sufficiently large radii, and the doubling constant may depend on the radius (see Definition \ref{defn:ls-doubling}).  Complete Riemannian manifolds with Ricci curvature bounded below form an important class of large-scale doubling mm-spaces.  Moreover, every finitely generated group equipped with the word metric and the counting measure is large-scale doubling.

We prove that under the large-scale doubling condition, large-scale equivalence coincides with weakly measured coarse equivalence (see Theorem \ref{thm:wmce-large-scale}).
Note that large-scale equivalence is not an equivalence relation for general mm-spaces.
In this sense, weakly measured coarse equivalence is a more satisfactory notion than large-scale equivalence.

The notion of (weakly) measured coarse equivalence provides a framework for studying large-scale properties of mm-spaces that depend simultaneously on the metric and the measure. It is therefore natural to seek homological invariants preserved under this equivalence.  One of the fundamental such invariants in coarse geometry is uniformly finite homology introduced by Block--Weinberger \cite{BW}.
Motivated by this, we define \emph{weighted $\ell^\infty$ homology}, which is a variant of uniformly finite homology, for a large-scale doubling mm-space.  We note that, unlike uniformly finite homology, weighted $\ell^\infty$ homology depends not only on the metric structure but also on the underlying measure.

To define weighted $\ell^\infty$ homology of a large-scale doubling mm-space $X$, we first define it for a discretization $\Gamma$ of $X$ and then we define the weighted $\ell^\infty$ homology of $X$ as that of $\Gamma$.  Later we prove that this homology is independent of the choice of a discretization.  For the definition of the chain complex on $\Gamma$,
a key ingredient is to consider the measure on the product space $\Gamma^{n+1}$, $n \ge 0$, defined by
\[
\mu_{\Gamma^{n+1}} := \sum_{(x_0,\dots,x_n) \in \Gamma^{n+1}} (\mu_\Gamma(x_0) \cdots \mu_\Gamma(x_n))^{\frac1{n+1}} \delta_{(x_0,\dots,x_n)}
\]
where $\mu_\Gamma$ is the underlying measure on $\Gamma$.
Note that the exponent $\frac1{n+1}$ is crucial here.
This choice exactly compensates for the change of exponent under the boundary operator.
If $\mu_\Gamma$ is the counting measure, then the resulting (weighted) $\ell^\infty$ homology coincides with the uniformly finite homology.
The weighted $\ell^\infty$ homology also generalizes the controlled coarse homology of Nowak--\v Spakula \cite{NS:cont}.
By contrast, if one uses the ordinary product measure (without the exponent $\frac{1}{n+1}$), the boundary operator does not preserve the resulting $\ell^\infty$ chain groups.  Thus, this construction does not yield a chain complex.

One of the main theorems in this paper is stated as follows.

\begin{thm} \label{thm:main1}
Let $X$ and $Y$ be large-scale doubling mm-spaces.
If $X \simeq_\wmce Y$, then
\[
H_n^{(\infty)}(X) \cong H_n^{(\infty)}(Y)
\]
for every $n \ge 0$.
\end{thm}

We next consider weighted amenability for mm-spaces.
Amenability of a metric space is defined by the failure of an isoperimetric inequality (see \cites{BW,NY:LSgeom}).
Mimicking this definition, we define weighted amenability for mm-spaces by the failure of an analogous isoperimetric inequality (see Definition \ref{defn:amenable}).
For a discrete metric space of bounded geometry equipped with the counting measure, weighted amenability coincides with amenability as a metric space (see Remark \ref{rem:BG-amenable}).

We have the following second main theorem.

\begin{thm} \label{thm:main2}
Let $X$ be a large-scale doubling mm-space.
Then the following are equivalent.
\begin{enumerate}
\item
The zeroth weighted $\ell^\infty$ homology group of $X$ vanishes.
\item
$X$ is weighted non-amenable.
\end{enumerate}
\end{thm}

Block--Weinberger \cite{BW} proved that, for a metric space admitting a quasi-lattice (i.e., a discrete net of bounded geometry), the vanishing of the zeroth uniformly finite homology is equivalent to non-amenability as a metric space.
Theorem \ref{thm:main2} generalizes this result to the weighted setting.

Weighted amenability depends essentially on the underlying measure, as demonstrated by the following example.  The real line $\R$ equipped with the Lebesgue measure $dx$ is weighted amenable, whereas $\R$ with the weighted measure $e^x dx$ is weighted non-amenable.  Correspondingly, the zeroth weighted $\ell^\infty$ homology group is nontrivial in the former case and vanishes in the latter.

Several additional examples are discussed in Section \ref{sec:examples}.

As a continuation of this paper, we prove a weighted version of Block--Weinberger--Whyte boundary criterion \cite{W:ame}, which shows a refinement of Theorem \ref{thm:main2} and also recovers a boundary criterion for the fundamental class in controlled coarse homology of Nowak--\v{S}pakula \cite{NS:cont}.
We publish this separately as \cite{OS:bdycri}.

\medskip\noindent
{\bf Idea of proof.}
We briefly explain the main ideas of the proofs.
To define the weighted $\ell^\infty$ homology of a large-scale doubling mm-space $X$, we first associate a discretization of $X$ that is measured coarsely equivalent to $X$, and then define the weighted $\ell^\infty$ homology to be that of the discretization.

A key ingredient in the proof is the introduction of a natural chain map associated with a measured coarse correspondence between two discrete mm-spaces, which we call the \emph{transport operator}.
Since a measured coarse correspondence is represented by a coupling of measures rather than by a map between spaces, the transport operator is constructed by means of the disintegration of measures.
A similar idea for the definition of the transport operator appears in \cite{GKMS}.

Using the transport operator, we prove that the weighted $\ell^\infty$ homology is invariant under measured coarse equivalence.
This shows, in particular, that the definition is independent of the choice of discretization.

The overall strategy follows the classical proof of the homotopy invariance of singular homology via the prism operator, although the argument is considerably more delicate because of the measure-theoretic nature of the transport operator.

Moreover, weighted $\ell^\infty$ homology is unchanged when the underlying measure is replaced by a uniformly comparable measure.
Combining this fact with the invariance under measured coarse equivalence, we obtain Theorem \ref{thm:main1}.

For the proof of Theorem \ref{thm:main2}, we consider a discretization $\Gamma$ of a large-scale doubling mm-space $X$.
The weighted $\ell^\infty$ homology of $X$ is defined to be that of $\Gamma$.
It is not difficult to prove that the weighted amenability of $X$ is equivalent to the weighted amenability of $\Gamma$.
Therefore, it suffices to prove the theorem for the discretization $\Gamma$.

The implication (1)$\implies$(2) for $\Gamma$ follows from a straightforward argument.

For the proof of the converse (2)$\implies$(1) for $\Gamma$, the overall strategy of the proof is similar to that of Block--Weinberger \cite{BW}.
In the argument of \cite{BW}, the vanishing of the fundamental homology class implies the vanishing of the zeroth homology group.  The key ingredient is the so-called Eilenberg swindle, carried out in the setting of homology with integral coefficients.  In our setting, however, weighted $\ell^\infty$ homology with integral coefficients is not available
because of the presence of weights.  Instead, we establish an analogue of the Eilenberg swindle for the weighted $\ell^\infty$ homology of $\Gamma$ with real coefficients by using disintegration of measures.  This allows us to prove the vanishing of the zeroth weighted $\ell^\infty$ homology of $\Gamma$.  In this way, we prove the implication (2)$\implies$(1) for $\Gamma$.

\medskip\noindent
{\bf Organization of the paper.}
In Section \ref{sec:prelim}, we review some preliminaries from measure theory and optimal transport theory.
In Section \ref{sec:mce}, we introduce (weakly) measured coarse equivalence and establish its basic properties.
In Section \ref{sec:L-infty-homology}, we define weighted $\ell^\infty$ homology for a discrete mm-space.
In Section \ref{sec:trans-op}, we introduce the transport operator and establish its basic properties.
In Section \ref{sec:inv-mBW}, we prove that the transport operator induces an isomorphism between weighted $\ell^\infty$ homology groups, thereby completing the proof of Theorem \ref{thm:main1}.
In Section \ref{sec:amenable}, we discuss weighted amenability and prove Theorem \ref{thm:main2}.
In Section \ref{sec:qi-measgr}, we prove that measure growth is invariant under weakly measured quasi-isometries.  We also show that subexponential measure growth implies weighted amenability.
In Section \ref{sec:examples}, we present several examples showing that amenability as a metric space and weighted amenability do not coincide.
Finally, in Section \ref{sec:appendix}, we study the relation between large-scale equivalence and weakly measured coarse equivalence, together with some additional results.

\begin{ack}
The authors used AI solely for English-language editing. The authors take full responsibility for the final version of the manuscript.
\end{ack}

\section{Preliminaries}
\label{sec:prelim}

In this section, we recall some definitions from measure theory and optimal transport theory and fix the notation used throughout the paper.

Let $X$ be a metric space with metric $d_X$.
A \emph{boundedly finite measure} on $X$ is a Borel measure on $X$ such that every bounded Borel set has finite measure.
Note that any boundedly finite measure is $\sigma$-finite.
Denote by $\cM_+(X)$ the set of boundedly finite measures on $X$.

Throughout this paper, by an \emph{mm-space} we mean a complete separable metric space equipped with a boundedly finite Borel measure with full support. For an mm-space $X$, we denote its metric and measure by $d_X$ and $\mu_X$, respectively.

For two mm-spaces $X$ and $Y$,
a \emph{coupling between $\mu_X$ and $\mu_Y$}
is a Borel measure $\sigma$ on $X \times Y$ such that $(p_X)_\bullet\sigma = \mu_X$ and $(p_Y)_\bullet\sigma = \mu_Y$,
where $p_X : X \times Y \to X$ and $p_Y : X \times Y \to Y$ denote the canonical projections,
and $(p_X)_\bullet\sigma$ and $(p_Y)_\bullet\sigma$ are the corresponding push-forward measures.
Denote by $\Pi(\mu_X,\mu_Y)$ the set of all couplings between $\mu_X$ and $\mu_Y$.

For a coupling $\sigma \in \Pi(\mu_X,\mu_Y)$, let $\{\tilde{\sigma}_x\}_{x \in X}$ be a disintegration of $\sigma$ with respect to the projection $p_X : X \times Y \to X$ (cf.~\cite{Fl:meas-theory}*{452I Theorem}).
More precisely, $\{\tilde{\sigma}_x\}_{x\in X}$ is a family of Borel probability measures on $X\times Y$ satisfying the following properties:
\begin{enumerate}
\item
For every Borel set $A \subset X \times Y$, the function $X \ni x \mapsto \tilde{\sigma}_x(A)$ is Borel measurable.
\item
$\tilde{\sigma}_x((\{x\} \times Y)^c) = 0$ $\mu_X$-a.e.~$x \in X$.
\item
\[
\sigma = \int_X \tilde{\sigma}_x \, d\mu_X(x),
\]
that is,
\[
\sigma(A) = \int_X \tilde{\sigma}_x(A) \, d\mu_X(x),
\]
for every Borel set $A \subset X \times Y$.
\end{enumerate}
A disintegration $\{\tilde{\sigma}_x\}_{x\in X}$ is uniquely determined $\mu_X$-a.e.
Similarly, let $\{\tilde{\sigma}_y\}_{y \in Y}$ be a disintegration of $\sigma$ with respect to the projection $p_Y : X \times Y \to Y$.
We define $\sigma_x := (p_Y)_\bullet\tilde{\sigma}_x$ and
$\sigma_y := (p_X)_\bullet\tilde{\sigma}_y$.
Since
$\tilde{\sigma}_x = \delta_x \otimes \sigma_x$ for $\mu_X$-a.e.~$x \in X$ and $\tilde{\sigma}_y = \sigma_y \otimes \delta_y$ for $\mu_Y$-a.e.~$y \in Y$,
 $\{\delta_x \otimes \sigma_x\}_{x \in X}$ and $\{\sigma_y \otimes \delta_y\}_{y \in Y}$ are disintegrations of $\sigma$.
$\{\sigma_x\}_{x \in X}$ is a family of Borel probability measures on $Y$ satisfying the following.
\begin{enumerate}
\item
For every Borel set $A \subset Y$, the function $X \ni x \mapsto \sigma_x(A)$ is Borel measurable.
\item
We have
\[
d\sigma(x,y) = d\sigma_x(y) d\mu_X(x) = d\sigma_y(x) d\mu_Y(y)
\]
in the sense that, for every nonnegative Borel measurable function $f$ on $X \times Y$,
\[
\int_{X \times Y} f(x,y)\,d\sigma(x,y)
= \int_X \int_Y f(x,y)\,d\sigma_x(y) d\mu_X(x)
= \int_Y \int_X f(x,y)\,d\sigma_y(x) d\mu_Y(y).
\]
\end{enumerate}
$\{\sigma_y\}_{y \in Y}$ also has similar properties.
We also call $\{\sigma_x\}_{x \in X}$ and $\{\sigma_y\}_{y \in Y}$ disintegrations of $\sigma$ with respect to $p_X$ and $p_Y$, respectively.

Let $X$, $Y$ and $Z$ be mm-spaces.
For $\sigma \in \Pi(\mu_X,\mu_Y)$ and $\tau \in \Pi(\mu_Y,\mu_Z)$, we first define a Borel measure $\tau\bullet\sigma$ on $X \times Y \times Z$ by
\[
d(\tau\bullet\sigma)(x,y,z) := d\sigma_y(x) d\tau_y(z) d\mu_Y(y).
\]
Let $p_{XZ} : X \times Y \times Z \to X \times Z$ denote the canonical projection.
We then define
\[
\tau\circ\sigma := (p_{XZ})_\bullet(\tau\bullet\sigma)
= \int_Y \sigma_y \otimes \tau_y \, d\mu_Y(y).
\]
We call $\tau\circ\sigma$ (and also $\tau\bullet\sigma$) the \emph{gluing of $\sigma$ and $\tau$}.
Then $\tau\circ\sigma \in \Pi(\mu_X,\mu_Z)$,
$(p_{XY})_\bullet(\tau\bullet\sigma) = \sigma$ and
$(p_{YZ})_\bullet(\tau\bullet\sigma) = \tau$ (see \cite{Vl:oldnew}*{Chapter 1}).

\section{Measured Coarse Equivalence}
\label{sec:mce}

In this section, we introduce (weakly) measured coarse equivalence between mm-spaces and establish some basic properties.

Let $X$, $Y$ and $Z$ be sets.
For subsets $S \subset X \times Y$, $T \subset Y \times Z$ and $A \subset X$, we define
\begin{align*}
& T \bullet S := (X \times T) \cap (S \times Z),\qquad
T \circ S := p_{XZ}(T \bullet S),\quad
\\
& S^{-1} := \{\,(y,x) \mid (x,y) \in S\,\},\quad
S[A] := p_Y((A \times Y) \cap S),
\end{align*}
where $p_Y : X \times Y \to Y$ and $p_{XZ} : X \times Y \times Z \to X \times Z$ are the canonical projections.
The operations $T \circ S$ and $S^{-1}$ generalize the composition and inverse of maps, respectively.
Note that, even if $X$, $Y$, and $Z$ are topological spaces, the sets $S[A]$ and $T\circ S$ need not be Borel, even when $A$, $S$, and $T$ are Borel.

\begin{defn}[Coarse equivalence] \label{defn:coarse-equiv}
Let $X$ and $Y$ be metric spaces and $S \subset X \times Y$ a subset.
We say that $S$ is \emph{boundedly dense}
if there exists $C > 0$ such that the $C$-neighborhoods of $p_X(S)$ and $p_Y(S)$ coincide with $X$ and $Y$, respectively.
We say that $S$ is \emph{bornologous} if for every $r > 0$ there exists $R = R(S,r) > 0$ such that
\[
(x_i,y_i) \in S\ (i=1,2),\ d_X(x_1,x_2) \le r  \implies d_Y(y_1,y_2) \le R.
\]
We say that $S$ is \emph{bi-bornologous} if $S$ and $S^{-1}$ are both bornologous.
A \emph{coarse correspondence between $X$ and $Y$} is a boundedly dense and bi-bornologous subset of $X\times Y$.
We say that $X$ and $Y$ are \emph{coarsely equivalent} and write
\[
X \simeq_\ce Y
\]
if there exists a coarse correspondence between $X$ and $Y$.
\end{defn}

Coarse equivalence is an equivalence relation.

\begin{rem}
In much of the literature, coarse equivalence between $X$ and $Y$
is defined by requiring that there exist bornologous maps $f : X \to Y$ and $g : Y \to X$ such that $g \circ f$ is close to $\mathrm{id}_X$ and $f \circ g$ is close to $\mathrm{id}_Y$.
It is well known that this definition is equivalent to the one given in Definition \ref{defn:coarse-equiv}.
\end{rem}

\begin{defn}[Measured coarse equivalence]
Let $X$ and $Y$ be mm-spaces.
A \emph{measured coarse correspondence between $X$ and $Y$} is a coupling $\sigma \in \Pi(\mu_X,\mu_Y)$ whose support is a coarse correspondence between $X$ and $Y$.
We say that $X$ and $Y$ are \emph{measured coarsely equivalent} and write
\[
X \simeq_\mce Y
\]
if there exists a measured coarse correspondence between them.
\end{defn}

Observe that for every coupling $\sigma \in \Pi(\mu_X,\mu_Y)$, the images $p_X(\supp\sigma)$ and $p_Y(\supp\sigma)$ are dense in $X$ and $Y$, respectively.  Hence, $\supp\sigma$ is always boundedly dense.

The following proposition is immediate from the definitions.

\begin{prop}
If two mm-spaces are measured coarsely equivalent, then they are coarsely equivalent.
\end{prop}

To prove that measured coarse equivalence is an equivalence relation, we first prove the following.

\begin{lem} \label{lem:supp-tau-sigma}
Let $X$, $Y$ and $Z$ be mm-spaces, and let $\sigma \in \Pi(\mu_X,\mu_Y)$, $\tau \in \Pi(\mu_Y,\mu_Z)$.
Then we have
\[
\supp (\tau \circ \sigma) \subset \overline{ \supp \tau \circ \supp \sigma }.
\]
\end{lem}

\begin{proof}
For $A \subset X \times Z$, we set
\[
\iota_Y(A) := \{\,(x,y,z) \in X \times Y \times Z \mid (x,z) \in A,\ y \in Y\,\}.
\]
For any Borel subsets $S \subset X \times Y$ and $T \subset Y \times Z$, the definition of $T \circ S$ yields 
$\iota_Y(\overline{T \circ S}) \supset \iota_Y(T \circ S)
\supset (X \times T) \cap (S \times Z)$.
Hence, $\iota_Y((\overline{T \circ S})^c) = (\iota_Y(\overline{T \circ S}))^c
\subset ((X \times T) \cap (S \times Z))^c
= (X \times T^c) \cup (S^c \times Z)$.  Therefore,
\begin{align*}
&\tau \circ \sigma((\overline{T \circ S})^c)
= \tau \bullet \sigma(\iota_Y((\overline{T \circ S})^c))
\le \tau \bullet \sigma((X \times T^c) \cup (S^c \times Z)) \notag\\
&\le \tau \bullet \sigma(X \times T^c)
+ \tau \bullet \sigma(S^c \times Z)
= \tau(T^c) + \sigma(S^c). \notag
\end{align*}
Setting $S := \supp\sigma$ and $T := \supp\tau$, we obtain
$\tau \circ \sigma((\overline{T \circ S})^c) = 0$
because $\sigma(S^c) = \tau(T^c) = 0$.
Thus, $\supp(\tau\circ\sigma) \subset \overline{T \circ S}$.
This completes the proof.
\end{proof}

We omit the proof of the following, as it is straightforward.

\begin{lem} \label{lem:comp-bornologous}
Let $X$, $Y$, $Z$ be metric spaces.
If two subsets $S \subset X \times Y$ and $T \subset Y \times Z$ are both bi-bornologous, then $T \circ S$ is bi-bornologous.
\end{lem}

Combining Lemmas \ref{lem:supp-tau-sigma} and \ref{lem:comp-bornologous}, we prove the following.

\begin{lem} \label{lem:comp-mcecorr}
Let $X$, $Y$ and $Z$ be mm-spaces.
If $\sigma$ is a measured coarse correspondence between $X$ and $Y$ and $\tau$ is a measured coarse correspondence between $Y$ and $Z$, then $\tau\circ\sigma$ is a measured coarse correspondence between $X$ and $Z$.
\end{lem}

\begin{proof}
$\tau\circ\sigma$ is a coupling between $\mu_X$ and $\mu_Z$.
Therefore, it suffices to prove that $\supp(\tau \circ \sigma)$ is bi-bornologous.
By Lemma \ref{lem:comp-bornologous}, $\supp\tau \circ \supp\sigma$ is bi-bornologous and so is its closure.
Hence, by Lemma \ref{lem:supp-tau-sigma}, $\supp(\tau\circ\sigma)$ is bi-bornologous.
This completes the proof.
\end{proof}

\begin{prop} \label{prop:m-coarse-equiv}
Measured coarse equivalence between mm-spaces is an equivalence relation.
\end{prop}

\begin{proof}
We first prove reflexivity.
The push-forward of $\mu_X$ by the diagonal map $X \ni x \mapsto (x,x) \in X \times X$ is a measured coarse correspondence.
Hence, $X \simeq_\mce X$.

To prove symmetry, we assume that $X \simeq_\mce Y$.
Then there exists a measured coarse correspondence $\sigma$ between them.
The push-forward of $\sigma$ by the transpose map $X \times Y \ni (x,y) \mapsto (y,x) \in Y \times X$ is a measured coarse correspondence between $Y$ and $X$.
Therefore, $Y \simeq_\mce X$.

Transitivity follows from Lemma \ref{lem:comp-mcecorr}.

This completes the proof.
\end{proof}

\begin{defn}[Equivalence between (measured) coarse correspondences]
For two metric spaces $X$ and $Y$, we say that two coarse correspondences $S \subset X \times Y$ and $T \subset X \times Y$ are \emph{equivalent} if $S \cup T$ is a coarse correspondence (or equivalently, $S \cup T$ is bi-bornologous).

For two mm-spaces $X$ and $Y$, we say that two measured coarse correspondences $\sigma,\tau \in \Pi(\mu_X,\mu_Y)$ are \emph{equivalent} if $\supp\sigma$ and $\supp\tau$ are equivalent as coarse correspondences. 
\end{defn}

The proof of the following proposition is straightforward and will be omitted.

\begin{prop}
\begin{enumerate}
\item
Equivalence between coarse correspondences is an equivalence relation.
\item
Equivalence between measured coarse correspondences is an equivalence relation.
\end{enumerate}
\end{prop}

We next introduce a weaker variant of measured coarse equivalence.

For two Borel measures $\mu$ and $\nu$ on a topological space $X$,
we write $\mu \asymp \nu$ if there exists a constant $C \ge 1$ such that 
\[
C^{-1}\mu(A) \le \nu(A) \le C\mu(A)
\]
for every Borel set $A \subset X$.
The relation $\asymp$ is an equivalence relation.

\begin{defn}[Weakly measured coarse equivalence]
\label{defn:w-mce}
We say that two mm-spaces $X$ and $Y$ are \emph{weakly measured coarsely equivalent} and write
\[
X \simeq_\wmce Y
\]
if there exist measures $\nu_X \in \cM_+(X)$ and $\nu_Y \in \cM_+(Y)$ such that $\nu_X \asymp \mu_X$, $\nu_Y \asymp \mu_Y$ and $(X,\nu_X) \simeq_\mce (Y,\nu_Y)$.
\end{defn}

\begin{prop} \label{prop:wmce-equiv}
Weakly measured coarse equivalence is an equivalence relation.
\end{prop}

\begin{proof}
Reflexivity and symmetry are immediate.

To prove transitivity, we assume that $X \simeq_\wmce Y$ and $Y \simeq_\wmce Z$.
Then there exist measures $\nu_X \in \cM_+(X)$, $\nu_Y, \nu'_Y \in \cM_+(Y)$ and $\nu_Z \in \cM_+(Z)$ such that
$\mu_X \asymp \nu_X$, $\nu_Y \asymp \mu_Y \asymp \nu'_Y$ and $\mu_Z \asymp \nu_Z$,
and measured coarse correspondences $\sigma \in \Pi(\nu_X,\nu_Y)$ and $\tau \in \Pi(\nu'_Y,\nu_Z)$.
Let $\hat{\sigma} := \int_Y \tilde{\sigma}_y \, d\mu_Y(y)$.
Since $\sigma = \int_Y \tilde{\sigma}_y \, d\nu_Y(y)$,
we have $\hat{\sigma} \asymp \sigma$ because $\mu_Y\asymp\nu_Y$.
In particular, $\supp\hat{\sigma} = \supp\sigma$ and
$(p_X)_\bullet\hat{\sigma} \asymp (p_X)_\bullet\sigma = \nu_X$.
Since $(p_Y)_\bullet\hat{\sigma} = \mu_Y$,
$\hat{\sigma}$ is a measured coarse correspondence between $(X,(p_X)_\bullet\hat{\sigma})$ and $Y$.
In the same way, letting $\hat{\tau} := \int_Y \tilde{\tau}_y \, d\mu_Y(y)$
we have $(p_Z)_\bullet\hat{\tau} \asymp \nu_Z$ and so
$\hat{\tau}$ is a measured coarse correspondence between $Y$ and $(Z,(p_Z)_\bullet\hat{\tau})$.
Applying Lemma \ref{lem:comp-mcecorr}, we see that $\hat{\tau}\circ\hat{\sigma}$ is a measured coarse correspondence between $(X,(p_X)_\bullet\hat{\sigma})$ and $(Z,(p_Z)_\bullet\hat{\tau})$.
Therefore, $X \simeq_\wmce Z$.
\end{proof}

Note that the measure $\hat{\sigma}$ constructed in the proof of Proposition \ref{prop:wmce-equiv} is called a lift of $\mu_Y$ in \cite{GKMS}.

The proof of the following proposition is easy and omitted.

\begin{prop} \label{prop:bdd-1pt}
An mm-space is bounded {\rm(}i.e., of finite diameter{\rm)} if and only if it is weakly measured coarsely equivalent to a one-point mm-space.
\end{prop}

More generally, weakly measured coarse equivalence ignores bounded perturbations, as the following proposition shows.

\begin{prop}
Let $X$ be an mm-space, and let $Y \subset X$ be a nonempty subset such that $Y^c$ is a bounded open set.  Assume moreover that $\mu_X|_Y$ is fully supported on $Y$.  Then, $(Y,d_X|_{Y \times Y},\mu_X|_Y) \simeq_\wmce X$.
\end{prop}

\begin{proof}
If $X$ is bounded, the proposition follows from Proposition \ref{prop:bdd-1pt}.
Assume that $X$ is unbounded.
Since $Y^c$ is bounded, there exists a nonempty bounded open set $O$ contained in the interior of $Y$.
Let
\[
\sigma := (\id_Y,\id_Y)_\bullet(\mu_X|_Y) + \mu_X|_{Y^c} \otimes \mu_X|_O.
\]
Since
$\supp\sigma = \{\,(y,y) \mid y \in Y\,\} \cup \overline{(Y^c \times O)}$
is a coarse correspondence between $X$ and $Y$,
$\sigma$ is a measured coarse correspondence between
$(X,(p_X)_\bullet\sigma)$ and $(Y,(p_Y)_\bullet\sigma)$.
Moreover,
\[
(p_X)_\bullet\sigma = \mu_X|_Y + \mu_X(O) \mu_X|_{Y^c} \asymp \mu_X,
\quad
(p_Y)_\bullet\sigma = \mu_X|_Y + \mu_X(Y^c) \mu_X|_O \asymp \mu_X|_Y.
\]
Therefore, $X \simeq_\wmce (Y,d_X|_{Y\times Y},\mu_X|_Y)$.
\end{proof}

\section{Weighted $\ell^\infty$ Homology of Discrete Spaces}
\label{sec:L-infty-homology}

Let $X$ be a metric space.
\begin{defn}[Boundedly finite signed measure]
A \emph{boundedly finite signed measure} $\nu$ on $X$ is, by definition, a difference $\nu = \sigma - \tau$ of two boundedly finite measures $\sigma, \tau \in \cM_+(X)$.
For a Borel set $A \subset X$, the value $\nu(A)$ is defined whenever at least one of  $\sigma(A)$ and $\tau(A)$ is finite.

Denote by $\cM(X)$ the set of all boundedly finite signed measures on $X$.
\end{defn}

This notion differs slightly from the usual notion of a signed measure defined on a $\sigma$-algebra.
For a boundedly finite signed measure $\nu$ on $X$,
the value $\nu(A)$ is defined for every bounded Borel set $A \subset X$.

Let $\nu \in \cM(X)$ be a boundedly finite signed measure.
Since the restriction of $\nu$ to every bounded open set in $X$ is a signed measure in the ordinary sense, the local Jordan decompositions are compatible on overlaps by uniqueness and therefore glue together to yield a unique pair of mutually singular Borel measures $\nu^+, \nu^- \in \cM_+(X)$ such that $\nu = \nu^+ - \nu^-$.
We define $|\nu| := \nu^+ + \nu^-$ and $\supp\nu := \supp |\nu|$.

The set $\cM(X)$ naturally forms a vector space over $\R$.

For integers $n \ge 0$, $0 \le i_0 < i_1 < \cdots < i_k \le n$,
we define a map $p_{i_0i_1\dots i_k} : X^{n+1} \to X^{k+1}$ by
\[
p_{i_0i_1\dots i_k}(x_0,\dots,x_n) := (x_{i_0},\dots,x_{i_k}).
\]
We also define $p_{\hat{i}} := p_{0\cdots\hat{i}\cdots n} = p_{0\cdots i-1, i+1 \cdots n}$, i.e.,
\[
p_{\hat{i}}(x_0,\dots,x_n) = (x_0,\dots,\hat{x}_i,\dots,x_n)
= (x_0,\dots,x_{i-1},x_{i+1},\dots,x_n).
\]
For $\nu \in \cM(X^{n+1})$, we define
\[
(p_{i_0i_1\dots i_k})_\bullet\nu := (p_{i_0i_1\dots i_k})_\bullet\nu^+ - (p_{i_0i_1\dots i_k})_\bullet\nu^-.
\]

\begin{defn}[Condition BGR]
Let $\Gamma$ be a discrete mm-space.
We say that $\Gamma$ satisfies \emph{condition BGR} if
the following conditions BG and BR are satisfied.

\emph{Condition BG} (bounded geometry):
For every $r > 0$,
\[
\sup_{x \in \Gamma} \# B_r(x) < \infty,
\]
where $\# B_r(x)$ is the number of points in the closed $r$-ball $B_r(x)$ of $x$.

\emph{Condition BR} (bounded ratio):
For every $r > 0$ and $x,y \in \Gamma$ with $d_\Gamma(x,y) \le r$,
\[
\mu_\Gamma(x) \asymp_r \mu_\Gamma(y),
\]
where $A \asymp_r B$ means that there exists a constant $C_r \ge 1$ depending only on $r$ such that $C_r^{-1} A \le B \le C_r A$.

We call a discrete mm-space satisfying condition BGR a \emph{BGR discrete mm-space}.
\end{defn}

In the case where $X$ satisfies BG, $\nu \in \cM(X)$ if and only if $\nu = f\lambda$ for a function $f : X \to \R$, where $\lambda$ denotes the counting measure on $X$.

From now on, let $\Gamma$ be a BGR discrete mm-space.
Note that $\Gamma$ is at most countable.
Let $n \ge 0$ be an integer.
We equip the product space $\Gamma^{n+1}$ with the $\ell^1$ metric
\[
d_{\Gamma^{n+1}}(\mathbf{x},\mathbf{x}')
:= \sum_{i=0}^n d_\Gamma(x_i,x_i')
\]
for $\mathbf{x}=(x_0,\dots,x_n), \mathbf{x}'=(x_0',\dots,x_n') \in \Gamma^{n+1}$.
We also equip $\Gamma^{n+1}$ with the measure
\begin{equation} \label{eq:mu-Gamma-prod}
\mu_{\Gamma^{n+1}} := \sum_{(x_0,\dots,x_n) \in \Gamma^{n+1}} (\mu_\Gamma(x_0) \cdots \mu_\Gamma(x_n))^{\frac1{n+1}} \delta_{(x_0,\dots,x_n)}.
\end{equation}
Note that the exponent $\frac{1}{n+1}$ is crucial for ensuring that the boundary operator preserves the chain groups defined below.
Let $\ell^\infty(\Gamma^{n+1})$ denote the Banach space of bounded functions on $\Gamma^{n+1}$ with the $\ell^\infty$ norm $\|\cdot\|_\infty$.
For $\nu \in \cM(\Gamma^{n+1})$ we define
\[
\Delta_\nu := \sup_{(x_0,\dots,x_n) \in \supp \nu} \max_{i,j} d_\Gamma(x_i,x_j) \quad(\le \infty)
\]
if $\nu$ is not identically zero,
and $\Delta_\nu := 0$ if $\nu$ is identically zero.

\begin{defn}[$C^{(\infty)}_n(\Gamma)$] \label{defn:complex}
For a nonnegative integer $n$, let $C^{(\infty)}_n(\Gamma)$ denote the set of all $\nu \in \cM(\Gamma^{n+1})$ satisfying the following conditions \textup{($\Delta$)} and \textup{($\ell^\infty$)}.
\begin{align*}
\tag{$\Delta$}
&\Delta_\nu < \infty.\\
\tag{$\ell^\infty$}
&\left\| \frac{d\nu}{d\mu_{\Gamma^{n+1}}} \right\|_\infty < \infty.
\end{align*}
For a negative integer $n$, we define $C^{(\infty)}_n(\Gamma) := 0$.
\end{defn}

\begin{lem}
For every $n \in \Z$, $C^{(\infty)}_n(\Gamma)$ is a vector subspace of $\cM(\Gamma^{n+1})$.
\end{lem}

\begin{proof}
Let $\xi:=a\nu+b\eta$ for $\nu,\eta\in C^{(\infty)}_n(\Gamma)$ and $a,b\in\mathbb{R}$.
Condition $(\Delta)$ for $\xi$ follows from
$\supp\xi\subset \supp\nu\cup\supp\eta$.
Condition $(\ell^\infty)$ for $\xi$ follows immediately from linearity.
This proves $\xi\in C^{(\infty)}_n(\Gamma)$.
\end{proof}

We define a map $\Phi_\Gamma : C^{(\infty)}_n(\Gamma) \to \ell^\infty(\Gamma^{n+1})$ by
\begin{equation} \label{eq:Phi}
\Phi_\Gamma\nu :=  \frac{d\nu}{d\mu_{\Gamma^{n+1}}},
\qquad \nu \in C^{(\infty)}_n(\Gamma).
\end{equation}
Then $\Phi_\Gamma$ is an injective linear map.
Moreover, for $n=0$, condition $(\Delta)$ is automatic, since $\Delta_\nu=0$ for every $\nu \in C^{(\infty)}_0(\Gamma)$.
Therefore, $\Phi_\Gamma$ is an isomorphism if $n = 0$.

\begin{lem} \label{lem:Cinfty}
For every $n \ge 1$ and $0 \le i_0 < \cdots < i_k \le n$,
\[
(p_{i_0\cdots i_k})_\bullet C^{(\infty)}_n(\Gamma) \subset C^{(\infty)}_k(\Gamma).
\]
\end{lem}

\begin{proof}
It suffices to prove that
$(p_{\hat{i}})_\bullet C_n^{(\infty)}(\Gamma)
\subset C_{n-1}^{(\infty)}(\Gamma)$,
since the general statement then follows by repeated application of this fact.
We prove the case $i=0$; the other cases are analogous.

Take an arbitrary $\nu\in C_n^{(\infty)}(\Gamma)$.

We first show that $(p_{\hat{0}})_\bullet\nu \in \cM(\Gamma^n)$.
Let $A\subset \Gamma^n$ be a bounded Borel set.
Since $\Delta_\nu<\infty$, the set $(\Gamma\times A)\cap\supp\nu$
is bounded.
Therefore,
\[
(p_{\hat{0}})_\bullet\nu^\pm(A)
=
\nu^\pm(\Gamma\times A)
=
\nu^\pm\bigl((\Gamma\times A)\cap\supp\nu\bigr) < \infty.
\]
Hence $(p_{\hat{0}})_\bullet\nu^\pm\in\cM_+(\Gamma^n)$,
and consequently
$(p_{\hat{0}})_\bullet\nu\in\cM(\Gamma^n)$.

Next, we verify condition $(\Delta)$ for $(p_{\hat{0}})_\bullet\nu$.
Suppose that there exists a point
\[
\mathbf{x} = (x_1,\dots,x_n) \in \supp\bigl((p_{\hat{0}})_\bullet\nu\bigr)
\]
such that $\max_{i,j} d_\Gamma(x_i,x_j) > \Delta_\nu$.
Since
\[
|\nu|(\Gamma\times \{\mathbf{x}\}) \ge |(p_{\hat{0}})_\bullet\nu|(\mathbf{x}) > 0,
\]
we have $(\Gamma\times \{\mathbf{x}\})\cap\supp\nu\neq\emptyset$, which contradicts the definition of $\Delta_\nu$.
Therefore,
$\Delta_{(p_{\hat{0}})_\bullet\nu} \le \Delta_\nu < \infty$.

We prove condition $(\ell^\infty)$ for $(p_{\hat 0})_\bullet\nu$.
Take $r>0$ with $\Delta_\nu<r$.
We write $\nu(x_0,\dots,x_n)$ for $\nu(\{(x_0,\dots,x_n)\})$.
For every $(x_1,\dots,x_n)\in\Gamma^n$, we have
\begin{align*}
\left| \frac{d(p_{\hat 0})_\bullet\nu}{d\mu_{\Gamma^n}}(x_1,\dots,x_n)\right|
&= \left| \frac{ \sum_{x_0\in\Gamma}\nu(x_0,x_1,\dots,x_n) }
{ (\mu_\Gamma(x_1)\cdots\mu_\Gamma(x_n))^{1/n} } \right|  \\
&\le \frac{ \sum_{x_0\in B_r(x_1)} |\nu(x_0,x_1,\dots,x_n)| }
{ (\mu_\Gamma(x_1)\cdots\mu_\Gamma(x_n))^{1/n} }.
\end{align*}
If $\nu(x_0,x_1,\dots,x_n) \neq 0$, then
$d_\Gamma(x_0,x_j)<r$ for every $j=1,\dots,n$.
By condition $\mathrm{BR}$, there exists a constant $\lambda_r \ge 1$ such that $\mu_\Gamma(x_0) \le \lambda_r \mu_\Gamma(x_j)$.  Hence,
\[
\mu_\Gamma(x_0) \le
\lambda_r \bigl(\mu_\Gamma(x_1)\cdots\mu_\Gamma(x_n)\bigr)^{1/n}.
\]
Therefore,
\begin{align*}
\left| \frac{d(p_{\hat 0})_\bullet\nu}{d\mu_{\Gamma^n}}(x_1,\dots,x_n) \right|
&\le \lambda_r^{1/(n+1)} \sum_{x_0\in B_r(x_1)}
\frac{|\nu(x_0,x_1,\dots,x_n)|}
{(\mu_\Gamma(x_0)\cdots\mu_\Gamma(x_n))^{1/(n+1)}} \\
&\le \lambda_r^{1/(n+1)}
\left\|\frac{d\nu}{d\mu_{\Gamma^{n+1}}}\right\|_\infty
\sup_{x\in\Gamma}\# B_r(x) <\infty,
\end{align*}
where $\|\cdot\|_\infty$ is the sup-norm.
This proves condition $(\ell^\infty)$ for $(p_{\hat 0})_\bullet\nu$.
\end{proof}

We next introduce some notation needed to describe the boundary operator algebraically.
For sets $S$ and $T$, let $\cF(S,T)$ denote the free vector space over $\R$ generated by all maps $f : S \to T$.
For sets $S$, $T$ and $U$ and for
$f = \sum_i a_i f_i \in \cF(S,T)$, $g = \sum_j b_j g_j \in \cF(T,U)$, we define
\[
g \circ f := \sum_{i,j} a_i b_j (g_j \circ f_i) \in \cF(S,U).
\]
For $f = \sum_{0 \le i_0 < \cdots < i_k \le n} a_{i_0\cdots i_k} p_{i_0\cdots i_k} \in \cF(\Gamma^{n+1},\Gamma^{k+1})$, $k \le n$ and $\nu \in C^{(\infty)}_n(\Gamma)$, we also define
\[
f_\bullet\nu := \sum_{0 \le i_0 < \cdots < i_k \le n} a_{i_0\cdots i_k} (p_{i_0\cdots i_k})_\bullet\nu.
\]
It follows from Lemma \ref{lem:Cinfty} that $f_\bullet : C^{(\infty)}_n(\Gamma) \to C^{(\infty)}_k(\Gamma)$ is a well-defined linear map.
For such $f$ and for
$g = \sum_{0 \le j_0 < \cdots < j_l \le k} b_{j_0\cdots j_l} p_{j_0\cdots j_l} \in \cF(\Gamma^{k+1}, \Gamma^{l+1})$, $l \le k$, we have
\[
(g \circ f)_\bullet = g_\bullet f_\bullet
\]
on $C^{(\infty)}_n(\Gamma)$.

\begin{defn}[Boundary operator]
For $n \ge 1$, we define
\[
\tilde{\partial} := \sum_{i=0}^n (-1)^i p_{\hat{i}} \in \cF(\Gamma^{n+1},\Gamma^n),
\]
where $p_{\hat{i}} : \Gamma^{n+1} \to \Gamma^n$, $i=0,\dots,n$, are the projections.  The \emph{boundary operator} $\partial : C^{(\infty)}_n(\Gamma) \to C^{(\infty)}_{n-1}(\Gamma)$ is defined by $\partial \nu := \tilde{\partial}_\bullet\nu$
for $n \ge 1$, and $\partial := 0$ for $n \le 0$.
\end{defn}

\begin{prop}
We have $\partial\partial = 0$.
In particular, $\{(C^{(\infty)}_n(\Gamma),\partial)\}_{n \in \Z}$ is a chain complex.
\end{prop}

\begin{proof}
Since $\partial\partial = (\tilde{\partial} \circ \tilde{\partial})_\bullet$,
it suffices to prove $\tilde{\partial} \circ \tilde{\partial} = 0$.
By the definition,
\[
\tilde{\partial} \circ \tilde{\partial} = \sum_{i=0}^n \sum_{j=0}^{n-1}
(-1)^{i+j} p_{\hat{j}} \circ p_{\hat{i}},
\]
which is equal to $0$ by a standard simplicial identity.
\end{proof}

\begin{defn}[Weighted $\ell^\infty$ homology]
We define the \emph{$n$-th weighted $\ell^\infty$ homology group} by
\[
H_n^{(\infty)}(\Gamma) := \frac{\Ker(\partial : C^{(\infty)}_n(\Gamma) \to C^{(\infty)}_{n-1}(\Gamma))}
{\Image(\partial : C^{(\infty)}_{n+1}(\Gamma) \to C^{(\infty)}_n(\Gamma))}.
\]
\end{defn}

\begin{rem} \label{rem:approx-Hme}
If a Borel measure $\nu_\Gamma$ on $\Gamma$ satisfies $\nu_\Gamma \asymp \mu_\Gamma$,
then
$C^{(\infty)}_n(\Gamma) = C^{(\infty)}_n(\Gamma,\nu_\Gamma)$ and hence
$H^{(\infty)}_n(\Gamma) = H^{(\infty)}_n(\Gamma,\nu_\Gamma)$.
\end{rem}

\begin{rem}
If $\mu_\Gamma$ is the counting measure,
then the uniformly finite homology of $\Gamma$ introduced by Block and Weinberger coincides with the weighted $\ell^\infty$ homology of $\Gamma$.
\end{rem}

\section{Transport Operator}
\label{sec:trans-op}

In this section, given a measured coarse correspondence $\sigma \in \Pi(\mu_\Gamma,\mu_\Lambda)$ between two BGR discrete mm-spaces $\Gamma$ and $\Lambda$,
we introduce the transport operator $\sigma_\sharp : C^{(\infty)}_n(\Gamma) \to C^{(\infty)}_n(\Lambda)$, and prove its basic properties.

Let $X$ and $Y$ be mm-spaces.
We equip the product space $X \times Y$ with the $\ell^1$ metric $d_{X \times Y}((x,y),(x',y')) := d_X(x,x') + d_Y(y,y')$.

\begin{lem}
$\Pi(\mu_X,\mu_Y) \subset \cM_+(X \times Y)$.
\end{lem}

\begin{proof}
Take an arbitrary $\sigma \in \Pi(\mu_X,\mu_Y)$.
For every bounded Borel set $A \subset X \times Y$,
since $A \subset p_X(A) \times Y$, we have
$\sigma(A) \le \sigma(\overline{p_X(A)} \times Y)
= \mu_X(\overline{p_X(A)})$, which is finite by the boundedness of $\overline{p_X(A)}$.
Therefore, $\sigma \in \cM_+(X \times Y)$.
\end{proof}

From now on, we equip product spaces such as $X^{n+1}$ and $Y^{n+1}$ with the $\ell^1$ metric.

\begin{lem} \label{lem:disint-prod}
For every $\sigma \in \Pi(\mu_X,\mu_Y)$ and $n \ge 0$, we have the following.
\begin{enumerate}
\item
$\sigma^{\otimes (n+1)} \in \Pi(\mu_X^{\otimes (n+1)},\mu_Y^{\otimes (n+1)})$, where we identify $(X \times Y)^{n+1}$ with $X^{n+1} \times Y^{n+1}$.
\item
The families
$\{\sigma_{x_0} \otimes \cdots
\otimes \sigma_{x_n}\}_{(x_0,\dots,x_n) \in X^{n+1}}$
and $\{\sigma_{y_0} \otimes \cdots
\otimes \sigma_{y_n}\}_{(y_0,\dots,y_n) \in Y^{n+1}}$
are disintegrations of $\sigma^{\otimes(n+1)}$ for $p_{X^{n+1}}$ and $p_{Y^{n+1}}$, respectively.
\end{enumerate}
\end{lem}

\begin{proof}
We prove (1).
For any Borel sets $A_i \subset X$, $i=0,\dots,n$,
\begin{align*}
&\sigma^{\otimes (n+1)}(A_0 \times \dots \times A_n \times Y^{n+1})
= \sigma(A_0 \times Y) \cdots \sigma(A_n \times Y)\\
&= \mu_X(A_0) \cdots \mu_X(A_n)
= \mu_X^{\otimes (n+1)}(A_0 \times \cdots \times A_n),
\end{align*}
which implies $(p_{X^{n+1}})_\bullet \sigma^{\otimes (n+1)} = \mu_X^{\otimes (n+1)}$.  In the same way, $(p_{Y^{n+1}})_\bullet \sigma^{\otimes (n+1)} = \mu_Y^{\otimes (n+1)}$.
This proves (1).

We prove (2).  By (1),
\begin{align*}
&d\sigma^{\otimes (n+1)}(x_0,\dots,x_n,y_0,\dots,y_n)\\
&= d\sigma(x_0,y_0) \cdots d\sigma(x_n,y_n)\\
&= d\sigma_{x_0}(y_0) \cdots d\sigma_{x_n}(y_n)
d\mu_X(x_0) \cdots d\mu_X(x_n)\\
&= d(\sigma_{x_0}\otimes \cdots \otimes \sigma_{x_n})(y_0,\dots,y_n)
d\mu_X^{\otimes (n+1)}(x_0,\dots,x_n).
\end{align*}
Hence, $\{\sigma_{x_0} \otimes \cdots
\otimes \sigma_{x_n}\}_{(x_0,\dots,x_n) \in X^{n+1}}$ is a disintegration of $\sigma^{\otimes(n+1)}$ for $p_{X^{n+1}}$.
The other is proved in the same way.
\end{proof}

Let $\Gamma$ and $\Lambda$ be BGR discrete mm-spaces.
Let $\sigma \in \Pi(\mu_\Gamma,\mu_\Lambda)$ and let $\{\sigma_x\}_{x \in \Gamma}$ be a disintegration of $\sigma$ for the projection $p_\Gamma : \Gamma \times \Lambda \to \Gamma$.
Note that a disintegration is unique since $\Gamma$ is discrete.
For $\mathbf{x} = (x_0,\dots,x_n) \in \Gamma^{n+1}$,
we define
\[
\sigma^{\otimes(n+1)}_{\mathbf{x}} := \sigma_{x_0} \otimes \cdots
\otimes \sigma_{x_n}.
\]

\begin{defn}[Transport operator]
For a (nonnegative) measure $\nu$ on $\Gamma^{n+1}$, we define a measure $\sigma_\sharp\nu$ on $\Lambda^{n+1}$ by
\[
\sigma_\sharp\nu :=
\int_{\Gamma^{n+1}} \sigma^{\otimes(n+1)}_{\mathbf{x}} \, d\nu(\mathbf{x}).
\]
For a chain $\nu \in C^{(\infty)}_n(\Gamma)$, we define
\[
\sigma_\sharp\nu := \sigma_\sharp\nu^+ - \sigma_\sharp\nu^-,
\]
where $\sigma_\sharp\nu(A)$ for $A \subset \Lambda^{n+1}$ is defined whenever at least one of $\sigma_\sharp\nu^-(A)$ and $\sigma_\sharp\nu^+(A)$ is finite.
We call $\sigma_\sharp$ the \emph{transport operator} induced by $\sigma$.
\end{defn}

Let $\tilde{\sigma}^{\otimes(n+1)}_{\mathbf{x}}
:= \delta_{\mathbf{x}} \otimes \sigma^{\otimes(n+1)}_{\mathbf{x}}$.
For a measure $\nu$ on $\Gamma^{n+1}$, since
\[
(p_{\Gamma^{n+1}})_\bullet\int_{\Gamma^{n+1}} \tilde{\sigma}^{\otimes(n+1)}_{\mathbf{x}} \, d\nu(\mathbf{x})
= \nu, \qquad
(p_{\Lambda^{n+1}})_\bullet \int_{\Gamma^{n+1}} \tilde{\sigma}^{\otimes(n+1)}_{\mathbf{x}} \, d\nu(\mathbf{x})
= \sigma_\sharp\nu,
\]
we have $\int_{\Gamma^{n+1}} \tilde{\sigma}^{\otimes(n+1)}_{\mathbf{x}} \, d\nu(\mathbf{x}) \in \Pi(\nu,\sigma_\sharp\nu)$.
Thus, $\sigma_\sharp\nu$ can be regarded as the transported measure of $\nu$ along the transport plan $\sigma^{\otimes (n+1)}$.

\begin{lem} \label{lem:supp-sigma}
For every coupling $\sigma \in \Pi(\mu_\Gamma,\mu_\Lambda)$,
\[
\bigcup_{x \in \Gamma} \{x\} \times \supp\sigma_x = \supp\sigma.
\]
\end{lem}

\begin{proof}
Since $\sigma(x,y)=\mu_\Gamma(x)\sigma_x(y)$,
\[
y\in \supp\sigma_x
\ \Longleftrightarrow\ \sigma_x(y)>0
\ \Longleftrightarrow\ \sigma(x,y)>0
\ \Longleftrightarrow\ (x,y)\in \supp\sigma.
\]
This completes the proof.
\end{proof}

By the next lemma, the transport operator defines a linear operator from $C^{(\infty)}_n(\Gamma)$ to $C^{(\infty)}_n(\Lambda)$ provided that $\sigma$ is a measured coarse correspondence.

\begin{lem} \label{lem:sigma-sharp-op-prime}
For every measured coarse correspondence $\sigma$ between $\Gamma$ and $\Lambda$, we have
\[
\sigma_\sharp C^{(\infty)}_n(\Gamma) \subset C^{(\infty)}_n(\Lambda).
\]
\end{lem}

\begin{proof}
Take an arbitrary $\nu \in C^{(\infty)}_n(\Gamma)$.

We first show that $\sigma_\sharp\nu \in \cM(\Lambda^{n+1})$.
Let $A \subset \Lambda^{n+1}$ be any bounded Borel set.
Since $\supp\sigma$ is bi-bornologous,
so is $\supp\sigma^{\otimes(n+1)} = (\supp\sigma)^{n+1}$.
Hence, $S := \overline{(\supp\sigma^{\otimes(n+1)})^{-1}[A]}$ is bounded.

Fix $\mathbf{x} \in \Gamma^{n+1}$ and suppose that $\supp\sigma^{\otimes(n+1)}_\mathbf{x} \cap A \neq \emptyset$.
Then we find $\mathbf{y} \in \supp\sigma^{\otimes(n+1)}_\mathbf{x} \cap A$.
By Lemma \ref{lem:supp-sigma}, $(\mathbf{x},\mathbf{y}) \in \supp\sigma^{\otimes(n+1)}$.
It follows from the definition of $S$ that
$(\mathbf{x},\mathbf{y})\in
\supp\sigma^{\otimes(n+1)}$ and $\mathbf{y}\in A$ imply $\mathbf{x}\in S$.
Therefore, for every $\mathbf{x} \in \Gamma^{n+1} \cap S^c$, we have $\supp\sigma^{\otimes(n+1)}_\mathbf{x} \cap A = \emptyset$ and hence
\[
\int_{S^c} \sigma^{\otimes(n+1)}_{\mathbf{x}}(A) \, d\nu^\pm(\mathbf{x}) = 0,
\]
which implies
\[
\sigma_\sharp\nu^\pm(A) = \int_{S} \sigma^{\otimes(n+1)}_{\mathbf{x}}(A) \, d\nu^\pm(\mathbf{x})
\le \nu^\pm(S) < \infty.
\]
This proves $\sigma_\sharp\nu \in \cM(\Lambda^{n+1})$.

We verify condition ($\Delta$) for $\sigma_\sharp\nu$.
Take any $\mathbf{y} \in \supp\sigma_\sharp\nu$.
We have
\[
\int_{\Gamma^{n+1}} \sigma^{\otimes (n+1)}_\mathbf{x}(\mathbf{y}) \, d|\nu|(\mathbf{x})
= (\sigma_\sharp\nu^+ + \sigma_\sharp\nu^-)(\mathbf{y})
\ge |\sigma_\sharp\nu|(\mathbf{y}) > 0.
\]
There is $\mathbf{x} \in \supp\nu$ such that $\sigma^{\otimes (n+1)}_\mathbf{x}(\mathbf{y}) > 0$.
By Lemma \ref{lem:supp-sigma}, $\mathbf{y} \in \supp\sigma^{\otimes (n+1)}_\mathbf{x}$ implies
$(\mathbf{x},\mathbf{y}) \in \supp \sigma^{\otimes (n+1)}$.
Setting $(x_0,\dots,x_n) := \mathbf{x}$ and $(y_0,\dots,y_n) := \mathbf{y}$, we have $(x_i,y_i) \in \supp\sigma$ for every $i$.
Since $\mathbf{x} \in \supp\nu$ implies $d_\Gamma(x_i,x_j) \le \Delta_\nu$,
and since $\supp\sigma$ is bornologous,
there exists $R > 0$ depending only on $\Delta_\nu$ and $\sigma$
such that $d_\Lambda(y_i,y_j) \le R$ for every $i$ and $j$.
This proves $\Delta_{\sigma_\sharp\nu} \le R$.

We verify condition ($\ell^\infty$) for $\sigma_\sharp\nu$.
For $n\ge0$, $f \in \ell^\infty(\Gamma^{n+1})$ and
$\mathbf{y}=(y_0,\dots,y_n)\in\Lambda^{n+1}$, we define
\[
(T_\sigma f)(\mathbf{y}) :=
\sum_{\mathbf{x}=(x_0,\dots,x_n) \in \Gamma^{n+1}}
f(\mathbf{x})
\prod_{i=0}^n \sigma_{y_i}(x_i)
\left(
\frac{\mu_\Lambda(y_0)\cdots\mu_\Lambda(y_n)}
{\mu_\Gamma(x_0)\cdots\mu_\Gamma(x_n)}
\right)^{\frac{n}{n+1}}.
\]
For every $\nu \in C^{(\infty)}_n(\Gamma)$, a direct computation gives
\[
\Phi_\Lambda(\sigma_\sharp\nu)=T_\sigma(\Phi_\Gamma\nu).
\]
It suffices to prove that $T_\sigma$ maps
$\ell^\infty(\Gamma^{n+1})$ into
$\ell^\infty(\Lambda^{n+1})$.
More precisely, for every $f\in \ell^\infty(\Gamma^{n+1})$,
\begin{equation} \label{eq:T-sigma}
\|T_\sigma f\|_\infty
\le C(n,\sigma,\Gamma,\Lambda)\|f\|_\infty.
\end{equation}
Put $S:=\supp\sigma$.
Since $S^{-1}$ is bornologous, there exists $R>0$ such that
if $(x,y),(x',y)\in S$, then $d_\Gamma(x,x')\le R$.
By condition BG, there exists $N_R<\infty$ such that
$\# B_R(x)\le N_R$ for every $x\in\Gamma$,
and by condition BR, there exists $\lambda_R\ge1$ such that
\[
d_\Gamma(x,x')\le R
\quad\Longrightarrow\quad
\mu_\Gamma(x')\le \lambda_R\mu_\Gamma(x).
\]
Fix $(x,y)\in S$.  Since $\sigma$ is a coupling,
\[
\mu_\Lambda(y)
=
\sum_{x':(x',y)\in S}\sigma(x',y)
\le
\sum_{x':(x',y)\in S}\mu_\Gamma(x').
\]
Every such $x'$ belongs to $B_R(x)$, and hence
$\mu_\Lambda(y) \le N_R\lambda_R\mu_\Gamma(x)$.
Thus there exists a constant $C_0>0$ such that
$\mu_\Lambda(y)\le C_0\mu_\Gamma(x)$
whenever $(x,y)\in S$.

Now suppose that the summand in the formula for
$(T_\sigma f)(\mathbf{y})$ is nonzero. Then
$(x_i,y_i)\in S$ for every $i$. Hence, for $i=0,\dots,n$,
$\frac{\mu_\Lambda(y_i)}{\mu_\Gamma(x_i)}\le C_0$.
Therefore
\[
\left(
\frac{\mu_\Lambda(y_0)\cdots\mu_\Lambda(y_n)}
{\mu_\Gamma(x_0)\cdots\mu_\Gamma(x_n)}
\right)^{\frac{n}{n+1}}
\le
C_0^n.
\]
Since each $\sigma_{y_i}$ is a probability measure, we obtain
\[
|(T_\sigma f)(\mathbf{y})|
\le
C_0^n
\sum_{\mathbf{x}}
|f(\mathbf{x})|
\prod_{i=0}^n\sigma_{y_i}(x_i)
\le
C_0^n\|f\|_\infty.
\]
This proves \eqref{eq:T-sigma}.
\end{proof}

\begin{lem} \label{lem:phati-sigmasharp}
Let $\sigma$ be a measured coarse correspondence between $\Gamma$ and $\Lambda$.
Then, for every $i=0,\dots,n$, we have
\[
(p_{\hat i})_\bullet \sigma_\sharp = \sigma_\sharp (p_{\hat i})_\bullet
\]
on $C^{(\infty)}_n(\Gamma)$.
\end{lem}

\begin{proof}
For every $\nu \in C^{(\infty)}_n(\Gamma)$ and function $f$ on $\Lambda^n$ with bounded support,
\begin{align*}
&\int_{\Lambda^n} f(y_0,\dots,\hat{y}_i,\dots,y_n) \,d(p_{\hat{i}})_\bullet\sigma_\sharp\nu(y_0,\dots,\hat{y}_i,\dots,y_n)\\
&= \int_{\Lambda^{n+1}} f(y_0,\dots,\hat{y}_i,\dots,y_n) \,d\sigma_\sharp\nu(y_0,\dots,y_n)\\
&= \int_{\Gamma^{n+1}} \int_\Lambda \cdots \int_\Lambda f(y_0,\dots,\hat{y}_i,\dots,y_n) \, d\sigma_{x_0}(y_0)\cdots \widehat{d\sigma_{x_i}(y_i)} \cdots  d\sigma_{x_n}(y_n)\\
&\qquad\qquad\qquad\qquad\qquad\qquad\qquad\qquad\quad d\nu(x_0,\dots,x_n)\\
&= \int_{\Lambda^n} f \, d\sigma_\sharp(p_{\hat{i}})_\bullet\nu.
\end{align*}
This completes the proof.
\end{proof}

The following is a direct consequence of Lemma \ref{lem:phati-sigmasharp}.

\begin{prop} \label{prop:bdy-sigma-sharp}
For every measured coarse correspondence $\sigma$ between $\Gamma$ and $\Lambda$,
\[
\partial \sigma_\sharp = \sigma_\sharp \partial.
\]
In particular, $\sigma_\sharp$ is a chain map.
\end{prop}

\begin{lem} \label{lem:disint-gluing}
Let $\Gamma$, $\Lambda$ and $Z$ be mm-spaces,
let $\sigma$ be a measured coarse correspondence between $\Gamma$ and $\Lambda$ and let $\tau$ be a measured coarse correspondence between $\Lambda$ and $Z$.  Then, the family of measures
\[
\int_\Lambda \tau_y \, d\sigma_x(y), \quad x \in \Gamma
\]
forms a disintegration of $\tau\circ\sigma$ with respect to $p_\Gamma : \Gamma \times Z \to \Gamma$.
\end{lem}

\begin{proof}
For every nonnegative function $f$ on $\Gamma \times Z$,
\begin{align*}
\int_{\Gamma \times Z} f(x,z) \, d(\tau\circ\sigma)(x,z)
&= \int_\Lambda \int_\Gamma \int_Z f(x,z) \, d\sigma_y(x) d\tau_y(z) d\mu_\Lambda(y)\\
&= \int_{\Gamma \times \Lambda} \int_Z f(x,z) \, d\tau_y(z) d\sigma(x,y)\\
&= \int_\Gamma \int_\Lambda \int_Z f(x,z) \, d\tau_y(z) d\sigma_x(y) d\mu_\Gamma(x).
\end{align*}
This completes the proof.
\end{proof}

\begin{prop} \label{prop:comp-sharp}
Let $\Gamma$, $\Lambda$ and $\Omega$ be BGR discrete mm-spaces.
Let $\sigma$ be a measured coarse correspondence between $\Gamma$ and $\Lambda$,
and let $\tau$ be a measured coarse correspondence between $\Lambda$ and $\Omega$.
Then
\[
(\tau\circ\sigma)_\sharp = \tau_\sharp\sigma_\sharp
\]
as homomorphisms from $C^{(\infty)}_n(\Gamma)$ to $C^{(\infty)}_n(\Omega)$.
\end{prop}

\begin{proof}
Take disintegrations $\{\sigma_x\}_{x\in \Gamma}$ and
$\{\tau_y\}_{y\in \Lambda}$ of $\sigma$ and $\tau$ with respect to
$p_\Gamma:\Gamma\times \Lambda\to \Gamma$ and $p_\Lambda:\Lambda\times \Omega\to \Lambda$, respectively.
Define a disintegration $\{(\tau\circ\sigma)_x\}_{x\in \Gamma}$ of
$\tau\circ\sigma$ by
\[
(\tau\circ\sigma)_x:=\int_\Lambda \tau_y\,d\sigma_x(y).
\]

We set, for $\mathbf{x} = (x_0,\dots,x_n) \in \Gamma^{n+1}$ and
$\mathbf{y} = (y_0,\dots,y_n) \in \Lambda^{n+1}$,
\begin{align*}
&\sigma^{\otimes(n+1)}_{\mathbf{x}}
:=
\sigma_{x_0}\otimes\cdots\otimes\sigma_{x_n},
\qquad
\tau^{\otimes(n+1)}_{\mathbf{y}}
:=
\tau_{y_0}\otimes\cdots\otimes\tau_{y_n},\\
&(\tau\circ\sigma)^{\otimes(n+1)}_{\mathbf{x}}
:= (\tau\circ\sigma)_{x_0} \otimes\cdots\otimes (\tau\circ\sigma)_{x_n}.
\end{align*}
Let $\nu\in C^{(\infty)}_n(\Gamma)$.
For every $\mathbf{x}\in \Gamma^{n+1}$, we have
\[
(\tau\circ\sigma)^{\otimes(n+1)}_{\mathbf{x}}
=
\int_{\Lambda^{n+1}}
\tau^{\otimes(n+1)}_{\mathbf{y}}\,
d\sigma^{\otimes(n+1)}_{\mathbf{x}}(\mathbf{y}).
\]
Hence,
\begin{align*}
(\tau\circ\sigma)_\sharp\nu
&=
\int_{\Gamma^{n+1}}
(\tau\circ\sigma)^{\otimes(n+1)}_{\mathbf{x}}\,
d\nu(\mathbf{x}) \\
&=
\int_{\Gamma^{n+1}}
\int_{\Lambda^{n+1}}
\tau^{\otimes(n+1)}_{\mathbf{y}}\,
d\sigma^{\otimes(n+1)}_{\mathbf{x}}(\mathbf{y})\,
d\nu(\mathbf{x}) \\
&= \int_{\Lambda^{n+1}} \tau^{\otimes(n+1)}_{\mathbf{y}} \, d(\sigma_\sharp\nu)(\mathbf{y})\\
&= \tau_\sharp\sigma_\sharp\nu.
\end{align*}
This completes the proof.
\end{proof}

For a measured coarse correspondence $\sigma$ between $\Gamma$ and $\Lambda$,
we define a linear operator $\sigma_* : H_n^{(\infty)}(\Gamma) \to H_n^{(\infty)}(\Lambda)$ by
\[
\sigma_*[\nu] = [\sigma_\sharp\nu], \quad \nu \in \Ker\partial.
\]

Proposition \ref{prop:comp-sharp} implies the following.

\begin{cor} \label{cor:comp-star}
Under the same assumptions as in Proposition \ref{prop:comp-sharp},
we have
\[
(\tau \circ \sigma)_* = \tau_* \sigma_*
\]
as homomorphisms from $H^{(\infty)}_n(\Gamma)$ to $H^{(\infty)}_n(\Omega)$.
\end{cor}

\section{Invariance of Weighted $\ell^\infty$ Homology}
\label{sec:inv-mBW}

In this section, we prove the invariance of weighted $\ell^\infty$ homology under weakly measured coarse equivalence.

Let $\Gamma$ and $\Lambda$ be BGR discrete mm-spaces.

\begin{lem} \label{lem:equiv-star}
Let $\sigma$ and $\tau$ be two measured coarse correspondences between mm-spaces $\Gamma$ and $\Lambda$.
If $\sigma$ and $\tau$ are equivalent, then $\sigma_\sharp$ and $\tau_\sharp$ are chain homotopic and $\sigma_* = \tau_*$.
\end{lem}

\begin{proof}
On $\{0,1\}$, we define the metric $d_{\{0,1\}}(a,b) := |a-b|$, $a,b \in \{0,1\}$, and the measure $\mu_{\{0,1\}} := (\delta_0 + \delta_1)/2$.
We equip the product space $\Gamma \times \{0,1\}$ with the $\ell^1$ metric $d_{\Gamma \times \{0,1\}} := d_\Gamma + d_{\{0,1\}}$ and the product measure $\mu_\Gamma \otimes \mu_{\{0,1\}}$.
It is immediate that $\Gamma \times \{0,1\}$ is a BGR discrete mm-space.
For $s=0,1$, we define a map $\iota_s : \Gamma \to \Gamma \times \{0,1\}$ by $\iota_s(x) := (x,s)$, $x \in \Gamma$.  Let
\[
\iota_s^{\times (n+1)}(x_0,\dots,x_n)
:= (\iota_s(x_0),\dots,\iota_s(x_n)) \qquad\text{and}\qquad
\bar{\iota}_s := (\iota_s^{\times (n+1)})_\bullet.
\]

We first prove that
\begin{equation} \label{eq:bar-iota}
\bar{\iota}_s C^{(\infty)}_n(\Gamma) \subset C^{(\infty)}_n(\Gamma \times \{0,1\}).
\end{equation}
For every $\nu \in C^{(\infty)}_n(\Gamma)$, we are going to prove
$\bar{\iota}_s\nu \in C^{(\infty)}_n(\Gamma \times \{0,1\})$.
Under the identification between 
$(\Gamma \times \{0,1\})^{n+1}$ and $\Gamma^{n+1} \times \{0,1\}^{n+1}$,
we have
\begin{equation} \label{eq:iota-sharp}
\bar{\iota}_s \nu = \nu \otimes \delta_s^{\otimes (n+1)}.
\end{equation}
Indeed, for every function $f$ on $\Gamma^{n+1} \times \{0,1\}^{n+1}$ with bounded support,
\begin{align*}
&\int_{\Gamma^{n+1} \times \{0,1\}^{n+1}} f(\mathbf{x},\mathbf{s})
\, d(\bar{\iota}_s\nu)(\mathbf{x},\mathbf{s})\\
&= \int_{\Gamma^{n+1}} f(\mathbf{x},(s,\dots,s))
 \, d\nu(\mathbf{x})\\
&= \int_{\Gamma^{n+1} \times \{0,1\}^{n+1}} f(\mathbf{x},\mathbf{s})
\, d(\nu \otimes \delta_s^{\otimes (n+1)})(\mathbf{x},\mathbf{s}).
\end{align*}
This proves \eqref{eq:iota-sharp}.

We verify ($\Delta$) for $\bar{\iota}_s\nu$.
For every $((x_0,s_0),\dots,(x_n,s_n)) \in \supp\bar{\iota}_s\nu$,
\eqref{eq:iota-sharp} implies $(x_0,\dots,x_n) \in \supp\nu$ and $s_0 = \cdots = s_n = s$.
Therefore, $d_{\Gamma \times \{0,1\}}((x_i,s_i),(x_j,s_j)) = d_\Gamma(x_i,x_j)$ for every $i$ and $j$, which proves
$\Delta_{\bar{\iota}_s\nu} = \Delta_\nu < \infty$.

Condition ($\ell^\infty$) for $\bar{\iota}_s\nu$ follows from \eqref{eq:iota-sharp}.

Thus, $\bar{\iota}_s\nu$ belongs to $C^{(\infty)}_n(\Gamma \times \{0,1\})$.
This proves \eqref{eq:bar-iota}.

Since $p_{\hat{i}} \circ \iota_s^{\times (n+1)} = \iota_s^{\times n} \circ p_{\hat{i}}$, we have $\partial\bar{\iota}_s = \bar{\iota}_s\partial$ and in particular, $\bar{\iota}_s : C^{(\infty)}_n(\Gamma) \to C^{(\infty)}_n(\Gamma \times \{0,1\})$ is a chain map.

We construct a chain homotopy between $\bar{\iota}_0$ and $\bar{\iota}_1$ in the following.
Define maps $h_{n,i} : \Gamma^{n+1} \to (\Gamma \times \{0,1\})^{n+2}$, $i=0,\dots,n$, by
\[
h_{n,i}(x_0,\dots,x_n) := ((x_0,0),\dots,(x_i,0),(x_i,1),\dots,(x_n,1)).
\]

We prove that
\begin{equation} \label{eq:hi}
(h_{n,i})_\bullet C^{(\infty)}_n(\Gamma) \subset C^{(\infty)}_{n+1}(\Gamma \times \{0,1\})
\end{equation}
for every $i=0,\dots,n$.
Take an arbitrary $\nu \in C^{(\infty)}_n(\Gamma)$.

If $A \subset (\Gamma \times \{0,1\})^{n+2}$ is a bounded Borel set,
then $h_{n,i}^{-1}(A)$ is a bounded Borel set in $\Gamma^{n+1}$ and
$(h_{n,i})_\bullet |\nu|(A) = |\nu|(h_{n,i}^{-1}(A)) < \infty$.
Therefore, $(h_{n,i})_\bullet\nu \in \cM((\Gamma \times \{0,1\})^{n+2})$.

We verify condition ($\Delta$) for $(h_{n,i})_\bullet\nu$.
Since $\supp (h_{n,i})_\bullet\nu \subset \overline{h_{n,i}(\supp\nu)}$,
for every $((x_0,s_0),\dots,(x_{n+1},s_{n+1})) \in \supp (h_{n,i})_\bullet\nu$,
there is $(x_0',\dots,x_n') \in \supp\nu$ such that
\[
d_{(\Gamma \times \{0,1\})^{n+2}}((x_0,s_0),\dots,(x_{n+1},s_{n+1}),h_{n,i}(x_0',\dots,x_n')) < 1.
\]
Then, we have
$s_0 = \cdots = s_i = 0$, $s_{i+1} = \cdots = s_{n+1} = 1$ and
\[
\sum_{j=0}^i d_\Gamma(x_j,x_j') + \sum_{k=i+1}^{n+1} d_\Gamma(x_k,x_{k-1}') < 1.
\]
Hence, $d_\Gamma(x_j,x_k) \le \Delta_\nu + 1$ for every $j,k=0,\dots,n+1$,
which implies $\Delta_{(h_{n,i})_\bullet\nu} \le \Delta_\nu + 1 < \infty$.

We verify condition ($\ell^\infty$) for $(h_{n,i})_\bullet\nu$.
Since $\Delta_\nu<\infty$, all coordinates of points in $\supp\nu$
are uniformly close to each other.  By condition BGR, the measures of
these points are uniformly comparable.  Hence the Radon--Nikodym
density of $(h_{n,i})_\bullet\nu$ with respect to
$\mu_{(\Gamma\times\{0,1\})^{n+2}}$ is bounded.

Thus we obtain $(h_{n,i})_\bullet\nu \in C^{(\infty)}_{n+1}(\Gamma \times \{0,1\})$.  The proof of \eqref{eq:hi} is completed.

We define $h_n \in \cF(\Gamma^{n+1},(\Gamma \times \{0,1\})^{n+2})$ by
\[
h_n := \sum_{i=0}^n (-1)^i h_{n,i}.
\]
Then, by \eqref{eq:hi}, $(h_n)_\bullet C^{(\infty)}_n(\Gamma) \subset C^{(\infty)}_{n+1}(\Gamma \times \{0,1\})$.
We prove that $(h_n)_\bullet : C^{(\infty)}_n(\Gamma) \to C^{(\infty)}_{n+1}(\Gamma \times \{0,1\})$ is a chain homotopy between $\bar{\iota}_0$ and $\bar{\iota}_1$.
In fact, the standard prism calculation gives
\[
\tilde{\partial} h_n + h_{n-1} \tilde{\partial} = \iota_1^{\times (n+1)} - \iota_0^{\times (n+1)}.
\]
This implies
\begin{equation} \label{eq:equiv-star1}
\partial (h_n)_\bullet + (h_{n-1})_\bullet \partial = \bar{\iota}_1 - \bar{\iota}_0,
\end{equation}
that is, $(h_n)_\bullet$ is a chain homotopy between $\bar{\iota}_0$ and $\bar{\iota}_1$.

We define a measure $\gamma$ on $(\Gamma \times \Lambda) \times \{0,1\}$ by
\[
\gamma := \sigma \otimes \frac{\delta_0}{2} + \tau \otimes \frac{\delta_1}{2}.
\]
By identifying $(\Gamma \times \Lambda) \times \{0,1\}$ with $(\Gamma \times \{0,1\}) \times \Lambda$,
$\gamma$ is an element of $\Pi(\mu_\Gamma \otimes \mu_{\{0,1\}},\mu_\Lambda)$.
Define
\[
\gamma_{(x,0)}:=\sigma_x,\qquad
\gamma_{(x,1)}:=\tau_x.
\]
Then $\{\gamma_{(x,s)}\}_{(x,s) \in \Gamma \times \{0,1\}}$ is a disintegration of $\gamma$.
In fact,
\begin{align*}
d\gamma((x,s),y)
&= \frac12 (d\sigma(x,y) d\delta_0(s) + d\tau(x,y) d\delta_1(s))\\
&= (I_0(s) \, d\sigma_x(y) + I_1(s) \, d\tau_x(y)) d(\mu_\Gamma \otimes \mu_{\{0,1\}})(x,s)\\
&= d\gamma_{(x,s)}(y) d(\mu_\Gamma \otimes \mu_{\{0,1\}})(x,s).
\end{align*}
Since $\supp\gamma = (\supp\sigma\times\{0\})\cup(\supp\tau\times\{1\})$ and since $\supp\sigma$ and $\supp\tau$ are equivalent as coarse correspondences,
$\gamma$ is a measured coarse correspondence between $\Gamma \times \{0,1\}$ and $\Lambda$.

We next prove
\begin{equation} \label{eq:equiv-star2}
\gamma_\sharp\bar{\iota}_0 = \sigma_\sharp,
\qquad
\gamma_\sharp\bar{\iota}_1 = \tau_\sharp.
\end{equation}
For every $s=0,1$, $\nu \in C^{(\infty)}_n(\Gamma)$ and function $f$ on $\Gamma^{n+1} \times \{0,1\}^{n+1}$ with bounded support,
under the identification between $(\Gamma \times \{0,1\})^{n+1}$ and $\Gamma^{n+1} \times \{0,1\}^{n+1}$,
\begin{align*}
&\int_{\Gamma^{n+1} \times \{0,1\}^{n+1}} f(\mathbf{x},\mathbf{s})\, d\bar{\iota}_s\nu(\mathbf{x},\mathbf{s})\\
&= \int_{\Gamma^{n+1}} f(\mathbf{x},(s,\dots,s)) \, d\nu(\mathbf{x})\\
&= \int_{\Gamma^{n+1} \times \{0,1\}^{n+1}} f(\mathbf{x},\mathbf{s})
\cdot 2^{n+1} I_{(s,\dots,s)}(\mathbf{s}) \, d(\nu \otimes \mu_{\{0,1\}}^{\otimes (n+1)})(\mathbf{x},\mathbf{s}),
\end{align*}
which implies
\[
d\bar{\iota}_s\nu(\mathbf{x},\mathbf{s}) = 2^{n+1} I_{(s,\dots,s)}(\mathbf{s}) \, d(\nu \otimes \mu_{\{0,1\}}^{\otimes (n+1)})(\mathbf{x},\mathbf{s}).
\]
Let $\sigma^{\otimes(n+1)}_{\mathbf{x}} := \sigma_{x_0} \otimes\cdots\otimes \sigma_{x_n}$,
$\tau^{\otimes(n+1)}_{\mathbf{x}} := \tau_{x_0} \otimes\cdots\otimes \tau_{x_n}$ and
$\gamma^{\otimes(n+1)}_{(\mathbf{x},\mathbf{s})}
:= \gamma_{(x_0,s_0)} \otimes\cdots\otimes \gamma_{(x_n,s_n)}$
for $\mathbf{x} = (x_0,\dots,x_n) \in \Gamma^{n+1}$ and $\mathbf{s} = (s_0,\dots,s_n) \in \{0,1\}^{n+1}$.
We have
\begin{align*}
\gamma_\sharp\bar{\iota}_s\nu
&= \int_{\Gamma^{n+1} \times \{0,1\}^{n+1}} 2^{n+1} I_{(s,\dots,s)}(\mathbf{s}) \gamma^{\otimes(n+1)}_{(\mathbf{x},\mathbf{s})} \, d(\nu \otimes \mu_{\{0,1\}}^{\otimes (n+1)})(\mathbf{x},\mathbf{s})\\
&= \int_{\Gamma^{n+1}} \gamma^{\otimes(n+1)}_{(\mathbf{x},(s,\dots,s))} \, d\nu(\mathbf{x}).
\end{align*}
Here, we see
\[
\gamma^{\otimes(n+1)}_{(\mathbf{x},(s,\dots,s))} =
\begin{cases}
\sigma^{\otimes(n+1)}_\mathbf{x} & \text{if $s=0$},\\
\tau^{\otimes(n+1)}_\mathbf{x} & \text{if $s=1$}.
\end{cases}
\]
This proves \eqref{eq:equiv-star2}.

Let $H := \gamma_\sharp h_\bullet$ and apply $\gamma_\sharp$ to both sides of \eqref{eq:equiv-star1}.  Then, by Proposition \ref{prop:bdy-sigma-sharp} and \eqref{eq:equiv-star2}, we obtain
\[
\partial H + H \partial = \tau_\sharp - \sigma_\sharp,
\]
that is, $H$ is a chain homotopy between $\sigma_\sharp$ and $\tau_\sharp$.
Therefore $\sigma_\sharp$ and $\tau_\sharp$ induce the same homomorphism
on homology, namely $\sigma_*=\tau_*$.
\end{proof}

Using Lemma \ref{lem:equiv-star}, we prove the following.

\begin{thm} \label{thm:c-eq-iso}
Let $\Gamma$ and $\Lambda$ be measured coarsely equivalent BGR discrete mm-spaces.
Then, for every measured coarse correspondence $\sigma \in \Pi(\mu_\Gamma,\mu_\Lambda)$ and for every $n \ge 0$, the induced homomorphism
\[
\sigma_* : H^{(\infty)}_n(\Gamma) \to H^{(\infty)}_n(\Lambda)
\]
is an isomorphism.
\end{thm}

\begin{proof}
Let $I_\Gamma := (\id_\Gamma,\id_\Gamma)_\bullet\mu_\Gamma$.
Then $I_\Gamma$ is a measured coarse correspondence between $\Gamma$ and itself and satisfies
$\supp I_\Gamma = \{\,(x,x) \mid x \in \Gamma\,\}$.

We prove that $\sigma^{-1} \circ \sigma$ is equivalent to $I_\Gamma$ as measured coarse correspondence.
By Lemma \ref{lem:supp-tau-sigma},
$\supp(\sigma^{-1} \circ \sigma) \subset \supp\sigma^{-1} \circ \supp\sigma$, where any set is closed because of discreteness of the space.
By Lemma \ref{lem:comp-bornologous}, $\supp\sigma^{-1} \circ \supp\sigma$ is bi-bornologous and hence a coarse correspondence.
Therefore, $\supp(\sigma^{-1}\circ\sigma)$ and $\supp\sigma^{-1} \circ \supp\sigma$ are equivalent as coarse correspondences.
Since $\supp I_\Gamma \subset \supp\sigma^{-1} \circ \supp\sigma$, they are equivalent as coarse correspondences and hence
$\supp(\sigma^{-1}\circ\sigma)$ and $\supp I_\Gamma$ are equivalent.
Therefore, $\sigma^{-1} \circ \sigma$ and $I_\Gamma$ are equivalent as measured coarse correspondences.

By Corollary \ref{cor:comp-star} and Lemma \ref{lem:equiv-star},
\[
(\sigma^{-1})_* \sigma_* = (\sigma^{-1} \circ \sigma)_*
= (I_\Gamma)_* = \id_{H^{(\infty)}_n(\Gamma)}.
\]
In the same way,
\[
\sigma_*(\sigma^{-1})_* = \id_{H^{(\infty)}_n(\Lambda)}.
\]
Therefore, $\sigma_* : H^{(\infty)}_n(\Gamma) \to H^{(\infty)}_n(\Lambda)$
is an isomorphism.
\end{proof}

\begin{defn}[Large-scale doubling condition] \label{defn:ls-doubling}
We say that an mm-space $X$ satisfies the \emph{large-scale doubling condition} if there exists $r_0 > 0$ such that, for each $r \ge r_0$, there exists a constant $C_r>0$ satisfying
\[
\mu_X(B_{2r}(x)) \le C_r\,\mu_X(B_r(x))
\]
for every $x\in X$.

An mm-space satisfying the large-scale doubling condition is called a \emph{large-scale doubling mm-space}.
\end{defn}

The large-scale doubling condition is equivalent to the following seemingly stronger condition:
there exists $r_0 > 0$ such that for every $R \ge r \ge r_0$, there exists a constant $C_{r,R} > 0$ satisfying
\begin{equation} \label{eq:CrR}
\mu_X(B_R(x)) \le C_{r,R} \, \mu_X(B_r(x))
\end{equation}
for every $x \in X$.
This follows by iterating the large-scale doubling condition.

By the Bishop-Gromov volume comparison theorem,
a complete Riemannian manifold with Ricci curvature bounded from below satisfies the large-scale doubling condition.
Moreover, every finitely generated group equipped with the word metric and the counting measure is large-scale doubling.

The following lemma shows that every large-scale doubling mm-space admits a measured coarsely equivalent BGR discretization.

\begin{lem} \label{lem:doubling-Gamma}
Let $X$ be a large-scale doubling mm-space.  Then, there exist a discrete net $\Gamma \subset X$ and a Borel measurable nearest point projection $\pi : X \to \Gamma$ such that if we equip $\Gamma$ with the metric $d_\Gamma := d_X|_{\Gamma^2}$ and the measure $\mu_\Gamma := \pi_\bullet\mu_X$, then the mm-space $\Gamma$ satisfies condition BGR and is measured coarsely equivalent to $X$.
\end{lem}

\begin{proof}
Let $r_0$ be the constant as in \eqref{eq:CrR}.
Let $\Gamma \subset X$ be a maximal $2r_0$-discrete set, i.e., $\Gamma$ is maximal with respect to inclusion among subsets such that any two distinct points have distance greater than $2r_0$.
Then, $B_{2r_0}(\Gamma) = X$, i.e., $\Gamma$ is a net.
There exists a Borel measurable nearest point projection $\pi : X \to \Gamma$ (cf.~the proof of \cite{Sy:mmg}*{Lemma 3.4}).  Since $(\id_X,\pi)_\bullet\mu_X$ is a measured coarse correspondence between $X$ and $\Gamma$, we have $X \simeq_\mce \Gamma$.
It remains to prove condition BGR for $\Gamma$.

We verify condition BG.  For every $r > 0$ and $x \in X$,
the balls $B_{r_0}(y)$, $y \in B_r(x) \cap \Gamma$, are mutually disjoint and contained in $B_{r+r_0}(x)$.
Hence,
\[
\sum_{y \in B_r(x) \cap \Gamma} \mu_X(B_{r_0}(y)) \le \mu_X(B_{r+r_0}(x)).
\]
Let $C$ be the constant $C_{r_0,2r+r_0}$ as in \eqref{eq:CrR}.
For every $y \in B_r(x) \cap \Gamma$,
\[
\mu_X(B_{r+r_0}(x)) \le \mu_X(B_{2r+r_0}(y)) \le C\mu_X(B_{r_0}(y)).
\]
Therefore, $\#(B_r(x) \cap \Gamma) \le C$, which proves condition BG.

We verify BR.  For every $x \in \Gamma$, since $B_{r_0}(x) \subset \pi^{-1}(x) \subset B_{2r_0}(x)$, we have
\[
\mu_X(B_{r_0}(x)) \le \mu_\Gamma(x) \le \mu_X(B_{2r_0}(x)).
\]
For every $r > 0$ and $x,y \in \Gamma$ with $d_\Gamma(x,y) \le r$,
\[
\mu_\Gamma(x) \le \mu_X(B_{2r_0}(x)) \le \mu_X(B_{r+2r_0}(y)).
\]
Setting $C' := C_{r_0,r+2r_0}$, we have
\[
\mu_X(B_{r+2r_0}(y)) \le C' \mu_X(B_{r_0}(y)) \le C' \mu_\Gamma(y).
\]
Therefore, $\mu_\Gamma(x) \le C' \mu_\Gamma(y)$.
This completes the proof.
\end{proof}

\begin{defn}
Let $X$ be an mm-space and $\Gamma$ a discrete mm-space.
We call $\Gamma$ a \emph{discretization of $X$} if $\Gamma$ is measured coarse equivalent to $X$.
\end{defn}

\begin{defn}[Weighted $\ell^\infty$ homology]
Let $X$ be a large-scale doubling mm-space.
By Lemma \ref{lem:doubling-Gamma}, there exists a BGR discretization $\Gamma$ of $X$.
By Theorem \ref{thm:c-eq-iso}, the weighted $\ell^\infty$ homology group $H^{(\infty)}_n(\Gamma)$ is independent of the choice of discretization $\Gamma$.
We define the \emph{weighted $\ell^\infty$ homology group of $X$} by
\[
H^{(\infty)}_n(X) := H^{(\infty)}_n(\Gamma).
\]
\end{defn}

Theorem \ref{thm:main1} is restated as follows.

\begin{thm}
Let $X$ and $Y$ be large-scale doubling mm-spaces.
If $X \simeq_\wmce Y$, then
\[
H^{(\infty)}_n(X) \cong H^{(\infty)}_n(Y)
\]
for every $n \ge 0$.
\end{thm}

\begin{proof}
There exist measures $\nu_X \in \cM_+(X)$ and $\nu_Y \in \cM_+(Y)$ such that $\mu_X \asymp \nu_X$, $\mu_Y \asymp \nu_Y$ and $(X,\nu_X) \simeq_\mce (Y,\nu_Y)$.
Choose the same discretization $\Gamma$ of $X$ for both
$(X,\mu_X)$ and $(X,\nu_X)$.
Let
$\mu_\Gamma:=\pi_\bullet\mu_X$ and $\nu_\Gamma:=\pi_\bullet\nu_X$,
where $\pi : X \to \Gamma$ is the associated nearest point map.
Since $\mu_X\asymp\nu_X$, we have $\mu_\Gamma\asymp\nu_\Gamma$.
Hence $H^{(\infty)}_n(\Gamma,\mu_\Gamma) = H^{(\infty)}_n(\Gamma,\nu_\Gamma)$.
Let $\Lambda$ be the same discretization of $Y$ for $(Y,\mu_Y)$ and $(Y,\nu_Y)$ and define $\mu_\Lambda$ and $\nu_\Lambda$ in the same manner.  Then, $H^{(\infty)}_n(\Lambda,\mu_\Lambda) = H^{(\infty)}_n(\Lambda,\nu_\Lambda)$.
Since 
\[
(\Gamma,\nu_\Gamma) \simeq_\mce (X,\nu_X) \simeq_\mce (Y,\nu_Y) \simeq_\mce (\Lambda,\nu_\Lambda),
\]
we have $H^{(\infty)}_n(\Gamma,\nu_\Gamma) \cong H^{(\infty)}_n(\Lambda,\nu_\Lambda)$.
Thus,
\[
H^{(\infty)}_n(X) \cong H^{(\infty)}_n(\Gamma,\mu_\Gamma) = H^{(\infty)}_n(\Gamma,\nu_\Gamma) \cong H^{(\infty)}_n(\Lambda,\nu_\Lambda) = H^{(\infty)}_n(\Lambda,\mu_\Lambda) \cong H^{(\infty)}_n(Y).
\]
This completes the proof.
\end{proof}

\section{Weighted Amenability}
\label{sec:amenable}

In this section, we study weighted amenability and prove Theorem \ref{thm:main2}.

\subsection{From Vanishing to Weighted Non-amenability}
In this subsection, we prove that the vanishing of the zeroth weighted $\ell^\infty$ homology implies weighted non-amenability.

\begin{defn}[Weighted amenability] \label{defn:amenable}
Let $X$ be an mm-space.
For a set $\Omega \subset X$ and $r > 0$, we define
\[
\partial_r \Omega := B_r(\Omega) \cap B_r(\Omega^c),
\]
where $B_r(\Omega) := \{\,x \in X \mid d_X(x,\Omega) \le r\,\}$, $d_X(x,\Omega) := \inf_{y \in \Omega} d(x,y)$.
$X$ is said to be \emph{weighted amenable} if for every $r > 0$ there exists a sequence $\{\Omega_n\}_{n=1}^\infty$ of bounded Borel sets in $X$ such that $\mu_X(\Omega_n) > 0$ and
\[
\lim_{n\to\infty} \frac{\mu_X(\partial_r \Omega_n)}{\mu_X(\Omega_n)} = 0.
\]
We call such a sequence $\{\Omega_n\}$ an \emph{$r$-F\o lner sequence}.
\end{defn}

The weighted non-amenability of $X$ is equivalent to the following isoperimetric inequality:
There exist $r, C > 0$ such that
\[
\mu_X(\Omega) \le C \mu_X(\partial_r\Omega)
\]
for every bounded Borel set $\Omega \subset X$.

\begin{rem} \label{rem:BG-amenable}
Let $\Gamma$ be a discrete mm-space satisfying condition BG and with the counting measure $\mu_\Gamma = \sum_{x \in \Gamma} \delta_x$ as the underlying measure.
Then, $\Gamma$ is weighted amenable if and only if it is amenable as a metric space.
\end{rem}

\begin{prop} \label{prop:c-eq-am}
Let $X$ and $Y$ be mm-spaces such that $X \simeq_\wmce Y$.
If $X$ is weighted amenable then so is $Y$.
\end{prop}

\begin{proof}
First observe that if $\nu_X \asymp \mu_X$, then
$X$ is weighted amenable if and only if $(X,\nu_X)$ is weighted amenable.
Thus, it suffices to prove the proposition in the case where $X \simeq_\mce Y$.
Let $\sigma$ be a measured coarse correspondence between $X$ and $Y$.
For any given $r > 0$, we set $r' := R(2r)$, where $R(\cdot)$ is in the definition of bornologous for $S = (\supp\sigma)^{-1}$.
Assume that $X$ is weighted amenable.
Then there is an $r'$--F\o lner sequence $\{U_i\}$.
Let $V_i := \overline{(\supp\sigma)[U_i]} \subset Y$.
Since $(U_i \times V_i) \cap \supp\sigma \supset (U_i \times Y) \cap \supp\sigma$,
\begin{equation} \label{eq:ViUi}
\mu_Y(V_i) = \sigma(X \times V_i)
\ge \sigma(U_i \times V_i) = \sigma(U_i \times Y) = \mu_X(U_i).
\end{equation}

We next prove
\begin{equation} \label{eq:supp-bdy}
(\supp\sigma)^{-1}[\partial_r V_i] \subset \partial_{r'}U_i.
\end{equation}
Take an arbitrary $x \in (\supp\sigma)^{-1}[\partial_r V_i]$.
There is $y \in \partial_r V_i$ such that $(x,y) \in \supp\sigma$.
There are $y_+ \in V_i$ and $y_- \in V_i^c$ such that $d_Y(y,y_+) \le 3r/2$ and $d_Y(y,y_-) \le 3r/2$.
By the definition of $V_i$, there is $y_+' \in (\supp\sigma)[U_i]$ such that $d_Y(y_+,y_+') \le r/2$.  Hence $d_Y(y,y_+') \le 2r$.
Since $V_i^c$ is open and $(\supp\sigma)[X]$ is dense in $Y$,
there is $y_-' \in V_i^c \cap (\supp\sigma)[X]$ such that $d_Y(y_-,y_-') \le r/2$.
Hence $d_Y(y,y_-') \le 2r$.
There is $x_+ \in U_i$ such that $(x_+,y_+') \in \supp\sigma$.
Since $d_Y(y,y_+') \le 2r$,
we have $d_X(x,x_+) \le R(2r) = r'$.  Hence, $x\in B_{r'}(U_i)$.
There is $x_- \in X$ such that $(x_-,y_-') \in \supp\sigma$.
Since $y_-' \notin (\supp\sigma)[U_i]$, we have $x_- \in U_i^c$.
Since $d_X(x,x_-) \le R(2r) = r'$, we see $x\in B_{r'}(U_i^c)$.
Thus, $x \in \partial_{r'}U_i$.  This proves \eqref{eq:supp-bdy}.

Since \eqref{eq:supp-bdy} implies
$(X \times \partial_rV_i) \cap \supp\sigma
= ((\supp\sigma)^{-1}[\partial_rV_i] \times \partial_rV_i) \cap \supp\sigma
\subset \partial_{r'}U_i \times Y$, we have
\[
\mu_Y(\partial_rV_i) = \sigma(X \times \partial_rV_i)
\le \sigma(\partial_{r'}U_i \times Y) = \mu_X(\partial_{r'}U_i).
\]
This together with \eqref{eq:ViUi} yields
\[
\frac{\mu_Y(\partial_rV_i)}{\mu_Y(V_i)}
\le \frac{\mu_X(\partial_{r'}U_i)}{\mu_X(U_i)},
\]
which converges to $0$ because $\{U_i\}$ is an $r'$--F\o lner sequence.
Thus, $\{V_i\}$ is an $r$-F\o lner sequence and
$Y$ is weighted amenable.
\end{proof}

\begin{lem} \label{lem:nonneg-chain}
Let $\Gamma$ be a BGR discrete mm-space.
Then, for every $\psi \in C^{(\infty)}_1(\Gamma)$, there exists $\psi'  \in C^{(\infty)}_1(\Gamma)$ such that $\psi' \ge 0$ and $\partial\psi = \partial\psi'$.
\end{lem}

\begin{proof}
Every $\psi \in C^{(\infty)}_1(\Gamma)$ is expressed as
\[
\psi = \sum_{(x_0,x_1) \in \Gamma^2} \psi(x_0,x_1) \delta_{(x_0,x_1)}.
\]
For each pair $x,y\in\Gamma$, we define $\psi'(x,y)$ as follows.
\begin{enumerate}
\item
we set $\psi'(x,x) := 0$,
\item
if $x \neq y$ and $\psi(x,y) - \psi(y,x) \ge 0$, then
we set $\psi'(x,y) := \psi(x,y) - \psi(y,x)$ and $\psi'(y,x) := 0$.
\item
if $x \neq y$ and $\psi(x,y) - \psi(y,x) < 0$, then
we set $\psi'(x,y) := 0$ and $\psi'(y,x) := \psi(y,x) - \psi(x,y)$.
\end{enumerate}
Then, we have $\psi' \ge 0$, $\psi' \in C^{(\infty)}_1(\Gamma)$ and
\begin{align*}
\partial\psi &= (p_1)_\bullet\psi - (p_0)_\bullet\psi
= \sum_{(x_0,x_1) \in \Gamma^2} \psi(x_0,x_1) \delta_{x_1}
- \sum_{(x_0,x_1) \in \Gamma^2} \psi(x_0,x_1) \delta_{x_0}\\
&= \sum_{(x,y) \in \Gamma^2} (\psi(x,y) - \psi(y,x)) \delta_y
= \sum_{(x,y) \in \Gamma^2} (\psi'(x,y) - \psi'(y,x)) \delta_y\\
&= \partial\psi'.
\end{align*}
\end{proof}

\begin{prop} \label{prop:H0-non-ame}
Let $\Gamma$ be a BGR discrete mm-space.
If $H^{(\infty)}_0(\Gamma)=0$, then $\Gamma$ is weighted non-amenable.
\end{prop}

\begin{proof}
Since $\mu_\Gamma\in C^{(\infty)}_0(\Gamma)$ and
$H^{(\infty)}_0(\Gamma)=0$, there exists
$\psi\in C^{(\infty)}_1(\Gamma)$ such that $\partial\psi=\mu_\Gamma$.
By Lemma~\ref{lem:nonneg-chain}, we may assume that $\psi\ge0$.

Take $r>0$ such that $\Delta_\psi<r$.
Put
$M:=\left\|\frac{d\psi}{d\mu_{\Gamma^2}}\right\|_\infty$.
Then, for every $(x,y)\in\Gamma^2$,
\[
\psi(x,y)\le M(\mu_\Gamma(x)\mu_\Gamma(y))^{1/2}.
\]
By condition BR, there exists $\lambda_r\ge1$ such that
$d_\Gamma(x,y) \le r$ implies
\[
\mu_\Gamma(x)\le \lambda_r\mu_\Gamma(y).
\]
By condition BG, set
\[
N_r:=\sup_{y\in\Gamma}\# B_r(y)<\infty.
\]

Let $U\subset\Gamma$ be a bounded Borel set with $\mu_\Gamma(U)>0$.
Since $\partial\psi=\mu_\Gamma$, we have
\begin{align*}
\mu_\Gamma(U) &=\partial\psi(U)
=(p_1)_\bullet\psi(U)-(p_0)_\bullet\psi(U)
=\psi(\Gamma\times U)-\psi(U\times\Gamma)\\
&=\psi(U^c\times U)-\psi(U\times U^c)
\le \psi(U^c\times U).
\end{align*}
Since $\Delta_\psi<r$, if $(x,y)\in\supp\psi\cap(U^c\times U)$, then
$y\in \partial_r U\cap U$. Hence
\[
\psi(U^c\times U)
\le
\psi(\Gamma\times \partial_r U).
\]
Therefore,
\begin{align*}
\mu_\Gamma(U)
&\le
\sum_{y\in\partial_r U} \sum_{x\in B_r(y)} \psi(x,y)
\le
M\sum_{y\in\partial_r U} \sum_{x\in B_r(y)}
(\mu_\Gamma(x)\mu_\Gamma(y))^{1/2}\\
&\le
M\lambda_r^{1/2} \sum_{y\in\partial_r U} \sum_{x\in B_r(y)}
\mu_\Gamma(y)
\le
M\lambda_r^{1/2}N_r\,\mu_\Gamma(\partial_r U).
\end{align*}
This is the isoperimetric inequality for weighted non-amenability.
Hence, $\Gamma$ is weighted non-amenable.
\end{proof}

\subsection{From Weighted Non-amenability to Vanishing}
To prove the converse implication, we argue by contraposition. We show that if $H^{(\infty)}_0(\Gamma)\neq0$, then $\Gamma$ is weighted amenable, thereby completing the proof of Theorem~\ref{thm:main2}.

\begin{lem} \label{lem:tail}
Let $\Gamma$ be a BGR discrete mm-space and let $\nu \in C^{(\infty)}_0(\Gamma)$ satisfy $[\nu] = 0$ in $H^{(\infty)}_0(\Gamma)$ and $\nu \ge 0$.  Then, for every $x \in \Gamma$ with $\nu(x) > 0$, there exists $t_x \in C^{(\infty)}_1(\Gamma)$ such that $t_x \ge 0$, $\partial t_x = \nu(x) \delta_x$ and $\sum_{x \in \Gamma : \nu(x) > 0} t_x \in C^{(\infty)}_1(\Gamma)$.
\end{lem}

\begin{proof}
For this $\nu$, there is $\psi \in C^{(\infty)}_1(\Gamma)$ such that $\nu = \partial\psi$.  By Lemma \ref{lem:nonneg-chain}, we may assume that $\psi \ge 0$, i.e., $\psi$ is a (nonnegative) measure on $\Gamma^2$.
Let $\varphi := (p_0)_\bullet\psi$.  Then, since $\nu = \partial\psi = (p_1)_\bullet\psi - \varphi$, we have $\psi \in \Pi(\varphi,\varphi + \nu)$.
Let $\{\tilde{\psi}_{x_1}\}_{x_1 \in \Gamma}$ be the disintegration of $\psi$ with respect to the projection $p_1$.

We first construct an increasing sequence of measures $\{\psi_i\}$, and then define $t_x$ as its limit.
Fix an arbitrary $x \in \Gamma$ with $\nu(x) > 0$.
We set $\varphi_0 := \nu(x)\delta_x$,
$\psi_1 := \int_\Gamma \tilde{\psi}_{x_1} \, d\varphi_0(x_1)$,
and $\varphi_1 := (p_0)_\bullet\psi_1$.
Since $(p_1)_\bullet\psi_1 = \int_\Gamma \delta_{x_1} \, d\varphi_0(x_1) = \varphi_0$, we have $\partial\psi_1 = \varphi_0 - \varphi_1$.
Since $0 \le \varphi_0 \le \nu \le \varphi + \nu = (p_1)_\bullet\psi$, we have $0 \le \psi_1 \le \int_\Gamma \tilde{\psi}_{x_1} \, d((p_1)_\bullet\psi)(x_1) = \psi$.  Hence, $\rho^1 := \psi - \psi_1$ is a (nonnegative) measure.
Using the disintegration of $\rho^1$, we set
$\psi_2 := \int_\Gamma \tilde{\rho}^1_{x_1} \, d\varphi_1(x_1)$ and
$\varphi_2 := (p_0)_\bullet\psi_2$.
We repeat this procedure.
Assume that $\psi_1,\dots,\psi_k$ are defined and define
$\varphi_i := (p_0)_\bullet\psi_i$, $i=1,\dots,k$.
Assume also that $\rho^k := \psi - \sum_{i=1}^k \psi_i \ge 0$
and define
\[
\psi_{k+1} := \int_\Gamma \tilde{\rho}^k_{x_1} \, d\varphi_k(x_1),
\qquad \varphi_{k+1} := (p_0)_\bullet\psi_{k+1}.
\]
Then, $\partial\psi_{k+1} = \varphi_k - \varphi_{k+1}$.

We prove
\begin{equation} \label{eq:psi-sum}
\rho^{k+1} := \psi - \sum_{i=1}^{k+1} \psi_i \ge 0.
\end{equation}
Since
\[
\rho^{k+1} = \rho^k - \psi_{k+1}
= \int_\Gamma \tilde{\rho}^k_{x_1} \,d(p_1)_\bullet\rho^k(x_1)
- \int_\Gamma \tilde{\rho}^k_{x_1} \, d\varphi_k(x_1),
\]
it suffices to prove that
\[
(p_1)_\bullet\rho^k - \varphi_k \ge 0.
\]
Since $(p_1)_\bullet\psi = \varphi + \nu$ and
$(p_1)_\bullet\psi_i = \int_\Gamma \delta_{x_1} \,d\varphi_{i-1}(x_1) = \varphi_{i-1}$, we have
\begin{align*}
(p_1)_\bullet\rho^k - \varphi_k &= (p_1)_\bullet\psi - \sum_{i=1}^k (p_1)_\bullet\psi_i - \varphi_k\\
&= \varphi + \nu - \sum_{i=0}^k \varphi_i = \nu - \varphi_0 + (p_0)_\bullet\rho^k \ge 0.
\end{align*}
This proves \eqref{eq:psi-sum}.

We set $t_{x,k} := \sum_{i=1}^k \psi_i$.
Since $t_{x,k}$ is monotone nondecreasing in $k$ and $0 \le t_{x,k} \le \psi$, it converges to some measure, say $t_x$.
Since $0 \le t_x \le \psi$, $t_x$ belongs to $C^{(\infty)}_1(\Gamma)$.
Since $\sum_{i=1}^k \varphi_i = \sum_{i=1}^k (p_0)_\bullet\psi_i \le (p_0)_\bullet\psi = \varphi$, the measure $\varphi_k$ converges to $0$ as $k\to\infty$.
Therefore,
$\partial t_{x,k} = \sum_{i=1}^k (\varphi_{i-1} - \varphi_i) = \varphi_0 - \varphi_k \to \varphi_0 = \nu(x)\delta_x$ as $k\to\infty$ and hence
\[
\partial t_x = (p_1)_\bullet t_x - (p_0)_\bullet t_x
= \lim_{k\to\infty} ((p_1)_\bullet t_{x,k} - (p_0)_\bullet t_{x,k})
= \lim_{k\to\infty} \partial t_{x,k} = \nu(x)\delta_x.
\]

Replacing $\nu$ by $\nu-\nu(x)\delta_x$ and $\psi$ by $\psi-t_x$, we repeat this argument to get $t_x$ for every $x \in \Gamma$ with $\nu(x) > 0$.
Since $\sum_{x \in \Gamma : \nu(x) > 0} t_x \le \psi$,
the measure $\sum_{x \in \Gamma : \nu(x) > 0} t_x$ belongs to $C^{(\infty)}_1(\Gamma)$.
This completes the proof.
\end{proof}

Using Lemma \ref{lem:tail}, we prove the following.

\begin{lem} \label{lem:H0}
Let $\Gamma$ be a BGR discrete mm-space.  The following are equivalent.
\begin{enumerate}
\item
$H^{(\infty)}_0(\Gamma) = 0$.
\item
$[\mu_\Gamma] = 0$ in $H^{(\infty)}_0(\Gamma)$.
\item
There exists $\nu \in C^{(\infty)}_0(\Gamma)$ such that
$\inf_\Gamma \frac{d\nu}{d\mu_\Gamma} > 0$ and
$[\nu] = 0$ in $H^{(\infty)}_0(\Gamma)$.
\end{enumerate}
\end{lem}

\begin{proof}
The implications (1)$\implies$(2)$\implies$(3) are immediate.

We prove (3)$\implies$(1).
Assume (3) holds.  Take an arbitrary $c \in C^{(\infty)}_0(\Gamma)$.
We apply Lemma \ref{lem:tail} to the chain $\nu$ in (3) to obtain
$t_x$, $x \in \Gamma$.
Since
\[
\frac{c(x)}{\nu(x)} = \frac{\frac{dc}{d\mu_\Gamma}(x)}{\frac{d\nu}{d\mu_\Gamma}(x)}
\]
is uniformly bounded,
we have $\sum_{x\in\Gamma} \frac{c(x)}{\nu(x)} t_x \in C^{(\infty)}_1(\Gamma)$.
Since $c = \sum_{x\in\Gamma} c(x)\delta_x
= \partial \sum_{x\in\Gamma} \frac{c(x)}{\nu(x)} t_x$,
we have $[c] = 0$ in $H^{(\infty)}_0(\Gamma)$.
This proves (1).
\end{proof}

\begin{lem} \label{lem:H0-am-Gamma}
Let $\Gamma$ be a BGR discrete mm-space.
If $H^{(\infty)}_0(\Gamma) \neq 0$, then $\Gamma$ is weighted amenable.
\end{lem}

\begin{proof}
The proof is similar to that in \cite{BW}.
Let $\cA := \{\,\varphi \in \ell^\infty(\Gamma) \mid \inf \varphi > 0\,\}$.
Assuming $H^{(\infty)}_0(\Gamma) \neq 0$, we have
$\cA \cap \Phi\partial C^{(\infty)}_1(\Gamma) = \emptyset$
by Lemma \ref{lem:H0}.
$\cA$ is an open convex set in $\ell^\infty(\Gamma)$ and $\Phi\partial C^{(\infty)}_1(\Gamma)$ is a linear subspace of $\ell^\infty(\Gamma)$.

By the Hahn--Banach separation theorem, there exists a continuous linear functional
$m : \ell^\infty(\Gamma) \to \mathbb{R}$ such that
$m>0$ on $\mathcal A$ and $m=0$ on
$\Phi\partial C^{(\infty)}_1(\Gamma)$.
Since the constant function $1$ belongs to $\mathcal A$, we have
$m(1)>0$.
Replacing $m$ by $m/m(1)$, we may assume that
$m(1)=1$.

We next prove that $m$ is positive.
Let $f\in \ell^\infty(\Gamma)$ satisfy $f\ge0$.
For every $\varepsilon>0$, we have
$f+\varepsilon\in\mathcal A$.
Hence, $m(f)+\varepsilon = m(f+\varepsilon) >0$.
Since $\varepsilon>0$ is arbitrary, we obtain
$m(f)\ge0$.
Therefore, $m$ is positive.

Since $m(1)=1$, it follows that
the operator norm of $m$ satisfies $\|m\|=1$.
Indeed, for every \(f\in\ell^\infty(\Gamma)\), positivity of \(m\) and
\(-\|f\|_\infty \le f\le\|f\|_\infty \) imply
\(|m(f)|\le\|f\|_\infty m(1)=\|f\|_\infty\); the reverse inequality follows by evaluating \(m\) at \(1\).

$m$ belongs to the continuous dual space $\ell^\infty(\Gamma)'$ of $\ell^\infty(\Gamma)$.

Let $\ell^1(\Gamma)$ denote the Banach space of $\mu_\Gamma$-integrable functions on $\Gamma$ with $\ell^1$ norm $\|\cdot\|_1$.
We prove the following.

\begin{clm}
There exists a net $\{m_\alpha\} \subset \ell^1(\Gamma)$ such that $\|m_\alpha\|_1 = 1$, $m_\alpha \ge 0$, $\supp m_\alpha$ is a finite set and $m_\alpha$ converges to $m$ in the weak $*$ topology.
\end{clm}

\begin{proof}
Let $K := \{\,\varphi \in \ell^1(\Gamma) \mid \|\varphi\|_1 = 1,\ \varphi \ge 0,\ \supp\varphi\ \text{is finite}\,\}$.
Embedding $\ell^1(\Gamma)$ into $\ell^\infty(\Gamma)'$,
we see that $K$ is a convex set in $\ell^\infty(\Gamma)'$.
It suffices to prove that $m$ is an element of the closure of $K$ with respect to the weak $*$ topology $\sigma(\ell^\infty(\Gamma)',\ell^\infty(\Gamma))$.
We prove this by contradiction.
Suppose that $m$ is not an element of the closure of $K$.
By the Hahn--Banach separation theorem for the weak $*$ topology, applied to the weak $*$
closure of $K$ and the point $m$, there exists $f \in \ell^\infty(\Gamma)$
such that
\[
\sup_{\varphi\in K} \int_\Gamma f\varphi\,d\mu_\Gamma < m(f).
\]
Since $\int_\Gamma \varphi\,d\mu_\Gamma = 1$ for every $\varphi\in K$
and $m(1)=1$, replacing $f$ by $f+C$ preserves the above strict inequality.
Adding a sufficiently large constant to $f$, we may assume that $f \ge 0$.
We have $\sup_{\varphi\in K} \int_\Gamma f \varphi \, d\mu_\Gamma = \sup f$.   On the other hand, since the operator norm of $m$ is equal to $1$, we have $m(f) \le \|f\|_\infty = \sup f$.  This is a contradiction.
\end{proof}

For $r > 0$, let $\Gamma^2(r) := \{ \,(x,y) \in \Gamma^2 \mid d_\Gamma(x,y) \le r\,\}$.

\begin{clm} \label{clm:varphi}
For every $\varepsilon, r > 0$, there exists $\varphi \in \ell^1(\Gamma)$ such that $\|\varphi\|_1 = 1$, $\varphi \ge 0$, $\supp\varphi$ is a finite set and
\[
\sum_{(x,y) \in \Gamma^2(r)} |\varphi(x) - \varphi(y)|\,(\mu_\Gamma(x)\,\mu_\Gamma(y))^{\frac12} < \varepsilon.
\]
\end{clm}

\begin{proof}
For $r > 0$ and a function $\psi : \Gamma^2 \to \R$, we define
\[
\|\psi\|_r := \int_{\Gamma^2(r)} |\psi | \, d\mu_{\Gamma^2} = \sum_{(x,y) \in \Gamma^2(r)} |\psi(x,y)| \,(\mu_\Gamma(x)\,\mu_\Gamma(y))^{\frac12}.
\]
This is a semi-norm on the set of functions on $\Gamma^2$.
Let $\cC := \{\,\psi : \Gamma^2 \to \R \mid \|\psi\|_r < \infty, \forall\,r > 0\,\}$.  Then, $\cC$ is a topological vector space by the family of semi-norms
$\{\|\cdot\|_r\}_{r > 0}$.
Since $\cC = \varprojlim \ell^1(\Gamma^2(r))$ (the projective limit with respect to $r\to\infty$), the continuous dual space $\cC'$ of $\cC$ satisfies
\[
\cC' \cong \varinjlim \ell^\infty(\Gamma^2(r)) = \Phi C^{(\infty)}_1(\Gamma)
\subset \ell^\infty(\Gamma^2).
\]

For $\varphi \in \ell^1(\Gamma)$, we define
\[
\bar{\delta}\varphi(x,y) := \varphi(y) - \varphi(x).
\]
Since there exists a constant $\lambda_r \ge 1$ such that $\mu_\Gamma(y) \le \lambda_r \mu_\Gamma(x)$, we have
\begin{align*}
\|\bar{\delta}\varphi\|_r &= \sum_{(x,y) \in \Gamma^2(r)} |\varphi(y)-\varphi(x)| \,(\mu_\Gamma(x)\,\mu_\Gamma(y))^{\frac12}
\le 2 \sum_{(x,y) \in \Gamma^2(r)} |\varphi(x)| \,(\mu_\Gamma(x)\,\mu_\Gamma(y))^{\frac12}\\
&\le 2 \lambda_r^{\frac12} \sum_{x \in \Gamma}  \# B_r(x)\, |\varphi(x)|\,\mu_\Gamma(x)
\le 2 \lambda_r^{\frac12} \|\varphi\|_1 \sup_{x \in \Gamma} \# B_r(x) < \infty,
\end{align*}
which implies $\bar{\delta}\varphi \in \cC$.
We define, for $\varphi_1 \in \ell^1(\Gamma)$ and $\varphi_2 \in \ell^\infty(\Gamma)$,
\[
\langle \varphi_1,\varphi_2 \rangle := \int_\Gamma \varphi_1 \varphi_2 \, d\mu_\Gamma
= \sum_{x\in\Gamma} \varphi_1(x)\,\varphi_2(x)\,\mu_\Gamma(x),
\]
and for $\psi_1 \in \cC$, $\psi_2 \in \ell^\infty(\Gamma^2)$,
\[
\langle \psi_1,\psi_2 \rangle := \int_{\Gamma^2} \psi_1 \psi_2 \, d\mu_{\Gamma^2} = \sum_{x,y\in\Gamma} \psi_1(x,y)\,\psi_2(x,y)\,(\mu_\Gamma(x)\,\mu_\Gamma(y))^{\frac12}.
\]
For every $\psi \in \Phi C^{(\infty)}_1(\Gamma)$,
setting $\bar{\partial}\psi := \Phi\partial\Phi^{-1}\psi$, we have, by a straightforward calculation,
\[
\bar{\partial}\psi(x) = \sum_{y \in \Gamma} (\psi(y,x)-\psi(x,y))
\left( \frac{\mu_\Gamma(y)}{\mu_\Gamma(x)} \right)^{\frac12}.
\]
Hence, for every $\varphi \in \ell^1(\Gamma)$,
\begin{align*}
\langle \varphi,\bar{\partial}\psi \rangle
&= \sum_{x\in\Gamma} \varphi(x) \sum_{y\in\Gamma} (\psi(y,x) - \psi(x,y)) (\mu_\Gamma(x)\mu_\Gamma(y))^{\frac12}\\
&= \sum_{x,y\in\Gamma} ( \varphi(y) - \varphi(x) ) \psi(x,y) (\mu_\Gamma(x)\,\mu_\Gamma(y))^{\frac12}\\
&= \langle \bar{\delta}\varphi,\psi \rangle.
\end{align*}

Since the net $\{m_\alpha\}$ converges to $m$ in the weak $*$ topology,
for every $\psi \in \Phi C^{(\infty)}_1(\Gamma)$,
\[
\langle \bar{\delta} m_\alpha,\psi \rangle
= \langle m_\alpha,\bar{\partial}\psi \rangle
\to m(\bar{\partial}\psi) = 0.
\]
Hence, $\bar{\delta}m_\alpha$ converges to $0$ in the weak topology of $\cC$.
In particular, $0$ belongs to the weak closure of $\bar{\delta}(K)$.
Since $\bar{\delta}(K)$ is convex and $\cC$ is a locally convex space,
the weak closure of $\bar{\delta}(K)$ coincides with its closure with
respect to the topology of $\cC$.
Therefore, for every $\varepsilon, r > 0$, there exists
$\varphi \in K$ such that
\[
\|\bar{\delta}\varphi\|_r < \varepsilon.
\]
This $\varphi$ satisfies the required condition.
\end{proof}

For every given $\varepsilon, r > 0$, let $\varphi$ be as in Claim \ref{clm:varphi}.
Let $t_0 = 0 < t_1 < \cdots < t_N$ be the values of $\varphi$ and let
$U_i := \{\varphi \ge t_i\}$.
Since
\[
\varphi = \sum_{i=1}^N (t_i - t_{i-1}) I_{U_i},
\]
setting $a_i := (t_i - t_{i-1}) \mu_\Gamma(U_i)$ we have
\[
\varphi = \sum_{i=1}^N \frac{a_i I_{U_i}}{\mu_\Gamma(U_i)}.
\]
Since $a_i > 0$ and $\|\varphi\|_1 = 1$, we have $\sum_{i=1}^N a_i = 1$.
Let $i(x)$ be such that $t_{i(x)} = \varphi(x)$ and let $i(x,y) := \min\{i(x),i(y)\}$ and $j(x,y) := \max\{i(x),i(y)\}$.
Then,
\begin{align*}
&\sum_{(x,y) \in \Gamma^2(r)} |\varphi(x) - \varphi(y)| \,(\mu_\Gamma(x)\,\mu_\Gamma(y))^{\frac12}
= \sum_{(x,y) \in \Gamma^2(r)} \sum_{i=i(x,y)+1}^{j(x,y)} \frac{a_i}{\mu_\Gamma(U_i)}\,(\mu_\Gamma(x)\,\mu_\Gamma(y))^{\frac12}\\
&= \sum_{i=1}^N \frac{a_i}{\mu_\Gamma(U_i)} \sum_{(x,y) \in \Gamma^2(r)} | I_{U_i}(x) - I_{U_i}(y)|\,(\mu_\Gamma(x)\,\mu_\Gamma(y))^{\frac12}.
\end{align*}
We are going to estimate this from below.
For every $x\in\partial_rU_i$, there is
$y\in B_r(x)$ lying on the opposite side of $U_i$, that is,
$|I_{U_i}(x)-I_{U_i}(y)|=1$.
Hence,
\[
\sum_{y\in B_r(x)} |I_{U_i}(x) - I_{U_i}(y)| \ge 1,
\]
which implies
\begin{align*}
&\sum_{(x,y) \in \Gamma^2(r)} | I_{U_i}(x) - I_{U_i}(y)|\,(\mu_\Gamma(x)\,\mu_\Gamma(y))^{\frac12}
\ge \lambda_r^{-\frac12} \sum_{(x,y) \in \Gamma^2(r)} | I_{U_i}(x) - I_{U_i}(y)|\,\mu_\Gamma(x) \\
&\ge \lambda_r^{-\frac12} \mu_\Gamma(\partial_rU_i).
\end{align*}
Therefore,
\begin{align*}
\varepsilon &> \sum_{(x,y) \in \Gamma^2(r)} |\varphi(x) - \varphi(y)| \,(\mu_\Gamma(x)\,\mu_\Gamma(y))^{\frac12}
\ge \lambda_r^{-\frac12} \sum_{i=1}^N a_i \frac{\mu_\Gamma(\partial_rU_i)}{\mu_\Gamma(U_i)}.
\end{align*}
Thus, there is $i$ such that
\[
\frac{\mu_\Gamma(\partial_rU_i)}{\mu_\Gamma(U_i)} < \lambda_r^{\frac12} \varepsilon.
\]
This proves that $\Gamma$ is weighted amenable.
\end{proof}

\begin{defn}[Ponzi scheme]
Let $\Gamma$ be a BGR discrete mm-space.
We call $\psi \in C^{(\infty)}_1(\Gamma)$ a \emph{Ponzi scheme}
if $\psi \ge 0$ and there exists a constant $c > 0$ such that $\partial\psi \ge c \mu_\Gamma$.
\end{defn}

Combining the previous results, we obtain the following.

\begin{thm} \label{thm:main-Gamma}
For a BGR discrete mm-space $\Gamma$, the following are equivalent.
\begin{enumerate}
\item
$H^{(\infty)}_0(\Gamma) = 0$.
\item
$\Gamma$ is weighted non-amenable.
\item
There exists a Ponzi scheme on $\Gamma$.
\end{enumerate}
\end{thm}

\begin{proof}
(1)$\implies$(2) follows from Proposition \ref{prop:H0-non-ame}.

(2)$\implies$(1) follows from Lemma \ref{lem:H0-am-Gamma}.

To prove (3) $\Longleftrightarrow$ (1), it suffices to prove the equivalence between (3) and Lemma \ref{lem:H0}(3).
The implication (3)$\implies$Lemma \ref{lem:H0}(3) is obvious.
Conversely, if we assume Lemma \ref{lem:H0}(3), then there are $\psi \in C^{(\infty)}_1(\Gamma)$ and a constant $c > 0$ such that $\partial\psi \ge c\mu_\Gamma$.  By Lemma \ref{lem:nonneg-chain}, we may assume $\psi \ge 0$.  Therefore, $\psi$ is a Ponzi scheme.  This completes the proof.
\end{proof}

\begin{proof}[Proof of Theorem \ref{thm:main2}]
The theorem follows from Proposition \ref{prop:c-eq-am} and Theorem \ref{thm:main-Gamma}.
\end{proof}

\begin{cor} \label{cor:Ric-amenable}
Let $M$ be a complete Riemannian manifold with Ricci curvature bounded from below.  We regard $M$ as an mm-space for the Riemannian distance and the Riemannian volume measure.
Then the following are equivalent.
\begin{enumerate}
\item
$H^{(\infty)}_0(M) = 0$.
\item
$M$ is weighted non-amenable.
\end{enumerate}
\end{cor}

\begin{proof}
Since $M$ satisfies the large-scale doubling condition, the corollary follows from Theorem \ref{thm:main2}. 
\end{proof}

\section{Quasi-Isometries and Measure Growth}
\label{sec:qi-measgr}

We recall the following standard definition.

\begin{defn}[Quasi-geodesic]
Let $X$ be a metric space and let $\lambda \ge 1$ and $k \ge 0$.
A map $\gamma : [\,a,b\,] \to X$ is called a \emph{$(\lambda,k)$-quasi-geodesic} if
\[
\frac{1}{\lambda} |t_1 - t_2| - k \le d_X(\gamma(t_1),\gamma(t_2))
\le \lambda |t_1 - t_2| + k.
\]
for every $t_1,t_2 \in [\,a,b\,]$.
$X$ is a \emph{quasi-geodesic space} if
there exist $\lambda \ge 1$ and $k \ge 0$ such that
any two points in $X$ can be joined by a $(\lambda,k)$-quasi-geodesic.
\end{defn}

The following is well-known (cf.~\cite{R:lect}*{Lemma 1.10}).

\begin{lem} \label{lem:quasi-corr-half}
Let $X$ be a quasi-geodesic space and $Y$ a metric space.
For a boundedly dense set $S \subset X \times Y$, the following are equivalent.
\begin{enumerate}
\item
There exist constants $C, L \ge 0$ such that
\[
d_Y(y_1,y_2) \le L \, d_X(x_1,x_2) + C
\]
for every $(x_i,y_i) \in S$, $i=1,2$.
\item
$S$ is bornologous.
\end{enumerate}
\end{lem}

\begin{defn}[Quasi-isometric correspondence]
Let $X$ and $Y$ be metric spaces.
A set $S \subset X \times Y$ is a \emph{quasi-isometric correspondence} if 
$S$ is boundedly dense and there exist constants $C, L \ge 0$ such that
\[
d_Y(y_1,y_2) \le L \,d_X(x_1,x_2) + C \quad\text{and}\quad
d_X(x_1,x_2) \le L \,d_Y(y_1,y_2) + C
\]
for every $(x_i,y_i) \in S$, $i=1,2$.
\end{defn}

Lemma \ref{lem:quasi-corr-half} directly implies the following.

\begin{prop}
Let $X$ and $Y$ be quasi-geodesic metric spaces and $S \subset X \times Y$ be a set.
Then, $S$ is a coarse correspondence if and only if
$S$ is a quasi-isometric correspondence.
\end{prop}

\begin{defn}[(Weakly) measured quasi-isometric]
Let $X$ and $Y$ be mm-spaces.
We say that $X$ and $Y$ are \emph{measured quasi-isometric} if there exists a coupling $\sigma \in \Pi(\mu_X,\mu_Y)$ such that
$\supp\sigma$ is a quasi-isometric correspondence.
$X$ and $Y$ are \emph{weakly measured quasi-isometric} if there exist
boundedly finite measures $\nu_X$ and $\nu_Y$ on $X$ and $Y$, respectively, with $\nu_X \asymp \mu_X$ and $\nu_Y \asymp \mu_Y$ such that $(X,\nu_X)$ and $(Y,\nu_Y)$ are measured quasi-isometric.
\end{defn}

\begin{defn}[Growth order of functions]
Let $f, g : [\,0,\infty\,) \to [\,0,\infty\,)$ be functions.
We define $f \preceq g$ if there exist constants $a,b,r_0 > 0$ such that
$f(r) \le a g(br)$ for every $r \ge r_0$.
Define $f \asymp g$ if $f \preceq g$ and $g \preceq f$ hold.
$\asymp$ is an equivalence relation.
We call the equivalence class of $f$ the \emph{growth order of $f$}.
\end{defn}

\begin{defn}[Measure growth]
Let $X$ be an mm-space.
For $x \in X$, the growth order of the function $[\,0,\infty\,) \ni r \mapsto \mu_X(B_r(x))$ is called the \emph{measure growth of $X$} (or \emph{of $\mu_X$}).
\end{defn}

Measure growth does not depend on $x$.
If $\mu_X \asymp \nu_X$, then the measure growth of $\mu_X$ coincides with that of $\nu_X$.

\begin{prop} \label{prop:meas-gr-inv}
Measure growth is invariant under weak measured quasi-isometry. In particular, when restricted to the category of quasi-geodesic mm-spaces, it is invariant under weak measured coarse equivalence.
\end{prop}

\begin{proof}
Assume that $X$ and $Y$ are measured quasi-isometric.
It suffices to prove that the measure growth of $X$ coincides with that of $Y$.
There is $\sigma \in \Pi(\mu_X,\mu_Y)$ such that
$\supp\sigma$ is a quasi-isometric correspondence.

Let $S := \supp\sigma$.
For every Borel set $A \subset X$, since
$(A \times Y) \cap S \subset X \times S[A]$, we have
\begin{equation} \label{eq:mu-SA}
\mu_Y(\overline{S[A]}) = \sigma(X \times \overline{S[A]})
\ge \sigma((A \times Y) \cap S) = \mu_X(A).
\end{equation}

Fix $(x_0,y_0) \in S$.
By \eqref{eq:mu-SA},
\[
\mu_X(B_r(x_0)) \le \mu_Y(\overline{S[B_r(x_0)]}).
\]

We next prove $\overline{S[B_r(x_0)]} \subset B_{Lr+C}(y_0)$,
where $L$ and $C$ are the constants in the definition of quasi-isometric correspondence.
Take an arbitrary $y \in \overline{S[B_r(x_0)]}$.
There is a sequence $\{y_i\}_{i=1}^\infty \subset S[B_r(x_0)]$ converging to $y$.
There is $x_i \in B_r(x_0)$ for each $i$ such that $(x_i,y_i) \in S$.
Since $S$ is a quasi-isometric correspondence,
\[
d_Y(y_0,y_i) \le L\,d_X(x_0,x_i) + C \le Lr + C.
\]
Hence, $d_Y(y_0,y) \le Lr + C$, which implies that
$y$ belongs to $B_{Lr+C}(y_0)$.
This proves $\overline{S[B_r(x_0)]} \subset B_{Lr+C}(y_0)$.

Thus we obtain
\[
\mu_X(B_r(x_0)) \le \mu_Y(B_{Lr+C}(y_0)).
\]
In the same way,
\[
\mu_Y(B_r(y_0)) \le \mu_X(B_{Lr+C}(x_0)).
\]
These inequalities imply that the measure growth of $X$ coincides with that of $Y$.
\end{proof}

\begin{prop} \label{prop:subexp}
If an mm-space $X$ has \emph{subexponential measure growth}, i.e., for some $x \in X$,
\[
\lim_{r\to\infty} \frac{\ln \mu_X(B_r(x))}{r} = 0,
\]
then $X$ is weighted amenable.
\end{prop}

\begin{proof}
We prove the contraposition.
Suppose that $X$ is weighted non-amenable.
There are constants $r, C > 0$ such that
\[
\mu_X(\Omega) \le C\mu_X(\partial_r\Omega)
\]
for every bounded Borel set $\Omega \subset X$.

Fix $x_0 \in X$, $\eta>0$ and $a_0>0$, and put
\[
a_n:=a_0+n(2r+\eta),
\qquad
R_n:=a_n+r+\eta
\]
for $n\ge0$.
By the triangle inequality,
\[
\partial_r B_{R_n}(x_0)
\subset
B_{a_{n+1}}(x_0)\setminus B_{a_n}(x_0).
\]
Indeed, if $z\in B_{a_n}(x_0)$, then
$d_X(z,B_{R_n}(x_0)^c)\ge r+\eta>r$, and hence
$z\notin\partial_r B_{R_n}(x_0)$.

Therefore,
\begin{align*}
\mu_X(B_{a_{n+1}}(x_0))
&\ge \mu_X(B_{a_n}(x_0))
 +\mu_X(\partial_r B_{R_n}(x_0))\\
&\ge \mu_X(B_{a_n}(x_0))
 +C^{-1}\mu_X(B_{R_n}(x_0))\\
&\ge (1+C^{-1})\mu_X(B_{a_n}(x_0)).
\end{align*}
It follows that
\[
\mu_X(B_{a_n}(x_0))
\ge (1+C^{-1})^n\mu_X(B_{a_0}(x_0)).
\]
Since $X$ has full support and $a_0>0$, we have
$\mu_X(B_{a_0}(x_0))>0$.
Put $q:=1+C^{-1}>1$.
For $t\in[a_n,a_{n+1})$ and for all sufficiently large $n$, monotonicity gives
\[
\frac{\ln\mu_X(B_t(x_0))}{t}
\ge
\frac{n\ln q+\ln\mu_X(B_{a_0}(x_0))}{a_{n+1}}.
\]
Since $a_{n+1}=a_0+(n+1)(2r+\eta)$, it follows that
\[
\liminf_{t\to\infty}
\frac{\ln\mu_X(B_t(x_0))}{t}
\ge
\frac{\ln(1+C^{-1})}{2r+\eta}>0.
\]
Thus, the measure growth of $X$ is not subexponential.
\end{proof}

\section{Examples}
\label{sec:examples}

In this section, we present several examples showing that amenability as a metric space and weighted amenability do not coincide.

We recall that for a graph equipped with the graph metric and the counting measure, weighted amenability coincides with amenability as a metric space.

\begin{ex}[$\Z$] \label{ex:Z}
We equip the set $\Z$ of integers with the Euclidean metric (i.e., the absolute value of the difference) and the measure
$\mu := \sum_{n \in \Z} 2^n \delta_n$.
Then, the measure $\psi := \sum_{n\in\Z} 2^{n+1} \delta_{(n+1,n)}$ on $\Z^2$ satisfies
\[
\frac{d\psi}{d\mu_{\Z^2}}(n+1,n) = \frac{2^{n+1}}{\sqrt{2^n 2^{n+1}}} = \sqrt{2},
\qquad \partial\psi = \mu,
\]
where $\mu_{\Z^2}$ denotes the measure defined by
\eqref{eq:mu-Gamma-prod} from $\mu$.
Therefore, $\psi$ is a Ponzi scheme and by Theorem \ref{thm:main-Gamma}, $(\Z,\mu)$ is weighted non-amenable and satisfies\\
$H^{(\infty)}_0(\Z,\mu) = 0$.  On the other hand, since the counting measure, denoted by $\mu_c$, has subexponential growth, Proposition \ref{prop:subexp} implies that $(\Z,\mu_c)$ is weighted amenable and satisfies $H^{(\infty)}_0(\Z,\mu_c) \neq 0$.
\end{ex}

\begin{ex}[$\R$]
Let $dx$ denote the one-dimensional Lebesgue measure on $\R$.
For the measure $\mu$ in Example \ref{ex:Z},
since $(\R,e^x dx) \simeq_\wmce (\Z,\mu)$,
$(\R,e^x dx)$ is weighted non-amenable and satisfies $H^{(\infty)}_0(\R,e^x dx) = 0$.
On the other hand, since $(\R,dx) \simeq_\wmce (\Z,\mu_c)$, $(\R,dx)$ is weighted amenable and satisfies $H^{(\infty)}_0(\R,dx) \neq 0$.
\end{ex}

\begin{ex}[Binary tree] \label{ex:tree}
Let $T$ be the rooted binary tree with root $O$,
i.e., $T$ is a connected graph containing no cycles and has exactly two edges emanating from $O$ and exactly three edges emanating from each vertex other than $O$.
We equip $T$ with the graph metric $d_T$.
Let $\rho(x) := d_T(O,x)$, $x \in T$.
Since the measure $\mu := \sum_{x \in T} 2^{-\rho(x)} \delta_x$ has subexponential growth, $(T,\mu)$ is weighted amenable and satisfies $H^{(\infty)}_0(T,\mu) \neq 0$.
On the other hand, for the counting measure $\mu_c$ on $T$, there exists a Ponzi scheme.  By Theorem \ref{thm:main-Gamma}, $(T,\mu_c)$ is weighted non-amenable and satisfies $H^{(\infty)}_0(T,\mu_c) = 0$.
\end{ex}

In order to prove the weighted non-amenability of a hyperbolic space,
we need the following.

\begin{defn}[Cheeger isoperimetric constant]
Let $X$ be an mm-space.
The \emph{Cheeger isoperimetric constant of $X$} is defined by
\[
h(X) := \inf_\Omega \frac{\mu_X^+(\Omega)}{\mu_X(\Omega)},
\]
where $\Omega$ runs over all bounded Borel sets in $X$ with $\mu_X(\Omega) > 0$ and we define
\[
\mu_X^+(\Omega) := \liminf_{\varepsilon\to 0+} \frac{\mu_X(B_\varepsilon(\Omega)) - \mu_X(\Omega)}{\varepsilon}.
\]
\end{defn}

\begin{prop} \label{prop:h}
If an mm-space $X$ satisfies $h(X) > 0$, then $X$ is weighted non-amenable.
\end{prop}

\begin{proof}
Assume $h(X) > 0$.
Let $\Omega \subset X$ be a bounded Borel set with $\mu_X(\Omega) > 0$ and set $\Omega_t := B_t(\Omega)$ for $t \ge 0$.
Since $\mu_X^+(\Omega_t)$ is the Dini derivative of $\mu_X(\Omega_t)$
and since $\mu_X^+(\Omega_t) \ge h(X) \mu_X(\Omega_t)$,
the Gronwall-type inequality (Lemma \ref{lem:Gronwall}) implies $\mu_X(\Omega_t) \ge e^{h(X)t} \mu_X(\Omega)$.
For every $0 < r' < r$,
since $\Omega_{r'} \subset \Omega \cup \partial_r\Omega$, we have
\[
e^{h(X)r'} \mu_X(\Omega) \le \mu_X(\Omega_{r'}) \le \mu_X(\Omega) + \mu_X(\partial_r\Omega).
\]
Taking the limit as $r' \nearrow r$ yields
\[
\mu_X(\Omega) \le \frac{1}{e^{h(X) r} - 1} \mu_X(\partial_r\Omega).
\]
Thus, $X$ is weighted non-amenable.
\end{proof}

\begin{cor} \label{cor:Had}
Let $M$ be an Hadamard manifold with sectional curvature $K \le -a^2$, where $a$ is a positive constant.  Then we have the following.
\begin{enumerate}
\item
$M$ is weighted non-amenable.
\item
Moreover, if the sectional curvature of $M$ is bounded from below, then $H^{(\infty)}_0(M) = 0$.
\end{enumerate}
\end{cor}

Note that the Ricci curvature of $M$ is bounded from below if and only if the sectional curvature of $M$ is bounded from below.

\begin{proof}
We prove (1).
By \cite{Y:isop}*{Prop.~3}, $h(M) \ge (n-1)a$.  This together with Proposition \ref{prop:h} implies that $M$ is weighted non-amenable.

(2) follows from (1) and Corollary \ref{cor:Ric-amenable}.
\end{proof}

\begin{ex}[Hyperbolic space] \label{ex:hyp-sp}
By Corollary \ref{cor:Had},
for $n \ge 2$, the $n$-dimensional complete simply connected hyperbolic space $H^n$ is weighted non-amenable and satisfies $H^{(\infty)}_0(H^n) = 0$.
On the other hand, if $\rho$ denotes the distance function from a fixed point in $H^n$, then the weighted volume measure $\nu := e^{-(n-1)\rho} \vol_{H^n}$ has subexponential growth and hence $(H^n,\nu)$ is weighted amenable and satisfies $H^{(\infty)}_0(H^n,\nu) \neq 0$.
\end{ex}

\section{Appendix} \label{sec:appendix}

Tessera \cite{Ts:large-sob} defined the following.

\begin{defn}[Large-scale equivalent]
We say that two large-scale doubling mm-spaces $X$ and $Y$ are \emph{large-scale equivalent} if there exists a map $F : X \to Y$ with the following properties.
\begin{enumerate}
\item[(a)]
For every $r > 0$ there exists $R = R(r) > 0$ such that for any $x,x' \in X$,
\begin{enumerate}
\item[(a1)] $d_X(x,x') \le r$ $\implies$ $d_Y(F(x),F(x')) \le R$,
\item[(a2)] $d_Y(F(x),F(x')) \le r$ $\implies$ $d_X(x,x') \le R$.
\end{enumerate}
\item[(b)] There exists a constant $C > 0$ such that $B_C(F(X)) = Y$.
\item[(c)] There exists $r_0 > 0$ such that for every $r \ge r_0$ and $x \in X$,
\[
\mu_X(B_r(x)) \asymp_r \mu_Y(B_r(F(x))).
\]
\end{enumerate}

Such a map $F : X \to Y$ is called a \emph{large-scale equivalence}.
\end{defn}

Letting $\graph F := \{\,(x,F(x)) \mid x \in X\,\}$, we note that
(a) is equivalent to the bi-bornologous property of $\graph F$, and that (b) is equivalent to the boundedly dense property of $\graph F$.

We prove the following.

\begin{thm} \label{thm:wmce-large-scale}
Let $X$ and $Y$ be large-scale doubling mm-spaces.
Then the following are equivalent.
\begin{enumerate}
\item $X \simeq_\wmce Y$.
\item
$X$ and $Y$ are large-scale equivalent.
\end{enumerate}
\end{thm}

To prove the theorem we need the following two lemmas.

\begin{lem}\label{lem:ball-volume}
Let $\sigma\in\Pi(\mu_X,\mu_Y)$ be a measured coarse correspondence.
Then, there exists $r_0 > 0$ such that for every $r \ge r_0$ and $(x,y)\in \supp\sigma$,
\[
\mu_X(B_r(x)) \asymp_r \mu_Y(B_r(y)).
\]
\end{lem}

\begin{proof}
Let $r_0$ be no smaller than the constants in \eqref{eq:CrR}
for both $X$ and $Y$.
Fix $r \ge r_0$ and let $(x,y)\in S := \supp\sigma$.
Since
$\supp\sigma\cap(B_r(x)\times Y)
\subset
B_r(x)\times\overline{S[B_r(x)]}$,
we have
\[
\mu_X(B_r(x))
=\sigma(B_r(x)\times Y)
\le
\sigma(X\times\overline{S[B_r(x)]})
=\mu_Y(\overline{S[B_r(x)]}).
\]
Since $S$ is bornologous, there exists $R=R(r)>0$ such that $R \ge r$ and
$S[B_r(x)]\subset B_R(y)$.
Hence
\[
\mu_X(B_r(x))
\le
\mu_Y(B_R(y))
\le
C_{r,R(r)} \mu_Y(B_r(y)),
\]
where the second inequality follows from the large-scale doubling property.

Applying the same argument to $S^{-1}$ completes the proof.
\end{proof}

\begin{lem} \label{lem:coarse-correspondence-map}
Let $X$ and $Y$ be metric spaces and let $S \subset X \times Y$ be a
coarse correspondence. Then there exists a map $F : X\to Y$ such that
$\graph F$ is a coarse correspondence
contained in a bounded neighborhood of $S$.
\end{lem}

\begin{proof}
Since $S$ is boundedly dense, there exists $C>0$ such that
$B_C(p_X(S))=X$ and $B_C(p_Y(S))=Y$. For every $x\in X$, choose
$(x^\ast,y^\ast)\in S$ such that $d_X(x,x^\ast)\le C$, and define
$F(x):=y^\ast$. Thus $(x^\ast,F(x))\in S$ and
\[
d_{X\times Y}\bigl((x,F(x)),(x^\ast,F(x))\bigr)
=d_X(x,x^\ast)\le C,
\]
where $X\times Y$ is equipped with the $\ell^1$ metric. Consequently,
$\graph F\subset B_C(S)$.

We first show that $\graph F$ is bornologous. Fix $r>0$ and suppose
$d_X(x_1,x_2) \le r$. Then
$d_X(x_1^\ast,x_2^\ast) \le r+2C$. Since
$(x_i^\ast,F(x_i))\in S$ for $i=1,2$ and $S$ is bornologous, there
exists $R>0$, depending only on $r$, such that
$d_Y(F(x_1),F(x_2)) \le R$. Hence $\graph F$ is bornologous.

We next show that $(\graph F)^{-1}$ is bornologous.  We fix $r>0$ and
suppose $d_Y(F(x_1),F(x_2)) \le r$. Since
$(F(x_i),x_i^\ast)\in S^{-1}$ for $i=1,2$ and $S^{-1}$ is
bornologous, there exists $R>0$, depending only on $r$, such that
$d_X(x_1^\ast,x_2^\ast) \le R$. Therefore
\[
d_X(x_1,x_2)
\le d_X(x_1,x_1^\ast)+d_X(x_1^\ast,x_2^\ast)
   +d_X(x_2^\ast,x_2)
\le 2C+R.
\]
Thus $(\graph F)^{-1}$ is bornologous, and hence $\graph F$ is
bi-bornologous.

We finally show that $\graph F$ is boundedly dense.
Let $y \in Y$. Since $B_C(p_Y(S)) = Y$, there exists
$(x,y^\ast) \in S$ such that $d_Y(y,y^\ast)\le C$. By the definition of
$F$, there exists $x^\ast\in X$ such that $d_X(x,x^\ast)\le C$ and
$(x^\ast,F(x))\in S$. Applying the bornologousness of $S$ at the scale
$C+1$, we obtain a constant $R>0$, independent of $x$ and $y$, such
that $d_Y(y^\ast,F(x)) \le R$. Hence
$d_Y(y,F(x)) \le C+R$, and therefore $B_{C+R}(F(X))=Y$.
\end{proof}

\begin{proof}[Proof of Theorem \ref{thm:wmce-large-scale}]
We prove (1)$\implies$(2).
Suppose that $X \simeq_\wmce Y$.
Then there exist Borel measures
$\nu_X\asymp\mu_X$ and $\nu_Y\asymp\mu_Y$
such that $(X,\nu_X) \simeq_\mce (Y,\nu_Y)$.
Let $\sigma\in\Pi(\nu_X,\nu_Y)$ be a measured coarse correspondence and put $S:=\supp\sigma$.

Let $F : X \to Y$ be the map constructed in the proof of Lemma \ref{lem:coarse-correspondence-map}.
$F$ satisfies {\rm (a)} and {\rm (b)}.
Let $C>0$ be the constant used in the construction of $F$ in the proof
of Lemma \ref{lem:coarse-correspondence-map}.
Choose a sufficiently large $r>0$ and $x\in X$.
There is $x'\in p_X(S)$ such that
$d_X(x,x')\le C$ and $(x',F(x))\in S$.
By Lemma~\ref{lem:ball-volume},
$\nu_X(B_r(x')) \asymp_r \nu_Y(B_r(F(x)))$.
Since $\nu_X \asymp \mu_X$ and $X$ is large-scale doubling,
$(X,\nu_X)$ is also large-scale doubling.
Hence $\nu_X(B_r(x)) \asymp_r \nu_X(B_r(x'))$.
Therefore $\nu_X(B_r(x)) \asymp_r \nu_Y(B_r(F(x)))$.
Since $\nu_X\asymp\mu_X$ and $\nu_Y\asymp\mu_Y$, we obtain
\[
\mu_X(B_r(x))
\asymp_r
\mu_Y(B_r(F(x))).
\]
Thus $F$ satisfies {\rm (c)}, and hence $X$ and $Y$ are large-scale equivalent.

We prove (2)$\implies$(1).
Let $F:X\to Y$ be a large-scale equivalence.
We divide $X$ into at most countably many Borel sets $B_i$
such that the diameter of $B_i$ is uniformly bounded.
Choose a point $x_i \in B_i$ for each $i$.
Define a map $F' : X \to Y$ by
$F'(x) := F(x_i)$ for $x \in B_i$.
Then $F'$ is Borel measurable and satisfies
$d_Y(F(x),F'(x)) \le C$ for every $x \in X$ and for a constant $C > 0$.  Since $F'$ is also a large-scale equivalence,
we replace $F$ by $F'$ and assume that $F$ is Borel measurable.

Let $R(r)$ be the constant for $r > 0$ as in (a)
and let $C>0$ be the constant in {\rm (b)}.
Let $r_0 > 0$ be so large that condition {\rm (c)} and the large-scale
doubling condition for both $X$ and $Y$ hold for every $r \ge r_0$.
Choose $s \ge r_0$ and $t \ge \max\{s,C+R(s)\}$.
For each $x\in X$, define a probability measure
\[
K_x:= \frac{\mu_Y|_{B_t(F(x))}}{\mu_Y(B_t(F(x)))}.
\]
Define a Borel measure $\sigma$ on $X\times Y$ by
\[
\sigma(A\times B) := \int_A K_x(B)\,d\mu_X(x).
\]
Since $K_x(Y)=1$, we have $(p_X)_\bullet\sigma=\mu_X$.
Let $\nu_Y:=(p_Y)_\bullet\sigma$. Then
$\nu_Y=h\,\mu_Y$, where
\[
h(y) := \int_X \frac{I_{B_t(F(x))}(y)}{\mu_Y(B_t(F(x)))}\,d\mu_X(x),
\]
and $I_B$ is the indicator function with respect to a set $B$.
Fix $y\in Y$.
If $x\in F^{-1}(B_t(y))$, then $d_Y(F(x),y) \le t$. Hence the large-scale doubling condition implies
$\mu_Y(B_t(F(x)))\asymp\mu_Y(B_t(y))$
uniformly for $x \in F^{-1}(B_t(y))$. Therefore
\[
h(y) \asymp \frac{1}{\mu_Y(B_t(y))} \int_X I_{B_t(F(x))}(y) \, d\mu_X(x)
=
\frac{\mu_X(F^{-1}(B_t(y)))}{\mu_Y(B_t(y))}.
\]

Choose $x_y\in X$ satisfying
$d_Y(F(x_y),y) \le C$.
If $z\in B_s(x_y)$, then
$d_Y(F(z),y)\le d_Y(F(z),F(x_y))+d_Y(F(x_y),y) \le R(s)+C \le t$.
Hence
\begin{equation} \label{eq:B-F}
B_s(x_y)\subset F^{-1}(B_t(y)).
\end{equation}

Conversely, if $z \in F^{-1}(B_t(y))$, then
$d_Y(F(z),F(x_y)) \le d_Y(F(z),y) + d_Y(y,F(x_y))$ $\le t+C$, and therefore
$d_X(z,x_y) \le R(t+C)$, which implies
\begin{equation} \label{eq:F-B}
F^{-1}(B_t(y)) \subset B_{R(t+C)}(x_y).
\end{equation}

Since $X$ is large-scale doubling and by \eqref{eq:B-F} and \eqref{eq:F-B},
$\mu_X(F^{-1}(B_t(y)))\asymp\mu_X(B_t(x_y))$.
By {\rm (c)}, $\mu_X(B_t(x_y))\asymp\mu_Y(B_t(F(x_y)))$.
Since $d_Y(F(x_y),y) \le C$, the large-scale doubling condition gives
$\mu_Y(B_t(F(x_y)))\asymp\mu_Y(B_t(y))$.
Hence we obtain
$\mu_X(F^{-1}(B_t(y))) \asymp \mu_Y(B_t(y))$, and consequently
$h(y)\asymp1$.
Therefore $\nu_Y = (p_Y)_\bullet\sigma\asymp\mu_Y$.

Put
$E_t:=\{(x,y)\in X\times Y\mid d_Y(F(x),y) \le t\}$.
Since $\sigma$ is concentrated on $E_t$,
we have $\supp\sigma\subset\overline{E_t}$.
Let $L:=R(1)$.
If $(x,y)\in\overline{E_t}$, then there exists a sequence
$(x_n,y_n)\in E_t$ converging to $(x,y)$.
For sufficiently large $n$, we have $d_X(x_n,x) \le 1$, and hence
\[
d_Y(F(x),y)
\le
d_Y(F(x),F(x_n))
+d_Y(F(x_n),y_n)
+d_Y(y_n,y)
\le
L+t+o(1).
\]
Thus
$\supp\sigma\subset
T := \{(x,y)\in X\times Y\mid d_Y(F(x),y)\le t+L\}$.
We claim that $T$ is bi-bornologous. Indeed, if
$(x_i,y_i)\in T$ for $i=1,2$ and $d_X(x_1,x_2) \le r$, then
\[
d_Y(y_1,y_2)
\le d_Y(y_1,F(x_1))+d_Y(F(x_1),F(x_2))
   +d_Y(F(x_2),y_2)
\le 2(t+L)+R(r),
\]
where $R(r)$ is given by the bornologousness of $\graph F$.
Conversely, if $d_Y(y_1,y_2) \le r$, then
$d_Y(F(x_1),F(x_2)) \le 2(t+L)+r$, and the bornologousness of
$(\graph F)^{-1}$ implies that $d_X(x_1,x_2)$ is bounded by a constant
depending only on $r$. Thus $T$ is bi-bornologous.
Hence $\supp\sigma$ is bi-bornologous.

Since $(p_X)_\bullet\sigma=\mu_X$ and
$(p_Y)_\bullet\sigma\asymp\mu_Y$, $\supp\sigma$ is boundedly dense.
Thus $\supp\sigma$ is a coarse correspondence.

Therefore $\sigma$ is a measured coarse correspondence between
$(X,\mu_X)$ and $(Y,\nu_Y)$.
Since $\nu_Y\asymp\mu_Y$, it follows that $X \simeq_\wmce Y$.
\end{proof}

The following lemma is needed in the proof of Proposition \ref{prop:h}.

\begin{lem}[Gronwall-type inequality] \label{lem:Gronwall}
Let $a\in\mathbb{R}$, and let
$f:[0,\infty)\to[0,\infty)$ be right-continuous and nondecreasing.
Suppose that
\[
f^+(t)
:=
\liminf_{h\downarrow0}
\frac{f(t+h)-f(t)}{h}
\ge af(t)
\]
for every $t\ge0$. Then
\[
f(t)\ge e^{at}f(0)
\]
for every $t\ge0$.
\end{lem}

\begin{proof}
Set
\[
g(t):=e^{-at}f(t).
\]
Since $f$ is right-continuous and nondecreasing, $g$ is
right-continuous, has left limits, and has no downward jumps:
\[
g(t-)\le g(t)
\]
for every $t>0$.

For $h>0$, set $A_h:= (f(t+h)-f(t))/h$.
Since $f$ is nondecreasing, $A_h\ge0$. Therefore,
\[
g^+(t) = e^{-at} \liminf_{h\downarrow0}
\left( e^{-ah}A_h + \frac{e^{-ah}-1}{h}f(t) \right)
= e^{-at}\bigl(f^+(t)-af(t)\bigr) \ge 0.
\]

We claim that $g$ is nondecreasing. Suppose, to the contrary, that
there exist $0\le s<t$ such that
\[
g(t)<g(s).
\]
Choose $\varepsilon>0$ such that
$g(t)+\varepsilon t<g(s)+\varepsilon s$,
and set
$q(u):=g(u)+\varepsilon u$.
Then $q$ is right-continuous, has left limits, has no downward jumps,
and
\[
q^+(u)=g^+(u)+\varepsilon\ge\varepsilon>0.
\]

We first note that $q$ attains its maximum on $[s,t]$. Indeed, let
\[
M:=\sup_{u\in[s,t]}q(u)
\]
and choose a sequence $\{u_k\}$ in $[s,t]$ such that $q(u_k)\to M$.
After passing to a subsequence, we may assume that $u_k\to u$ and
that $\{u_k\}$ is monotone. If $u_k\downarrow u$, right-continuity gives
$q(u)=M$. If $u_k\uparrow u$, then
$M=q(u-)\le q(u)\le M$,
and hence again $q(u)=M$.

Let $u_0$ be the largest point at which $q$ attains its maximum.
Such a point exists because, if $u_0$ is the supremum of the set of
maximizers, the absence of downward jumps implies $q(u_0)=M$.
Since $q(t)<q(s)$, we have $u_0<t$. By the maximality of $u_0$,
$q(u_0+h)<q(u_0)$ for every sufficiently small $h>0$. Hence
\[
q^+(u_0)
=
\liminf_{h\downarrow0}
\frac{q(u_0+h)-q(u_0)}{h}
\le0,
\]
contradicting $q^+(u_0)\ge\varepsilon$.

Thus $g$ is nondecreasing. Consequently,
$e^{-at}f(t)=g(t)\ge g(0)=f(0)$,
which proves the lemma.
\end{proof}

The next two propositions are the relation between BGR and large-scale doubling conditions for a discrete space.
Since their proofs are easy, we omit them.

\begin{prop}
Let $\Gamma$ be a BGR discrete mm-space.
Then $\Gamma$ satisfies the large-scale doubling condition with $r_0 = 0$.
\end{prop}


\begin{prop}
Let $\Gamma$ be a $r_0$-discrete mm-space for some
$r_0 > 0$.  If $\Gamma$ is large-scale doubling with $r_0$, then $\Gamma$ satisfies condition BGR.
\end{prop}

%
%


\end{document}